\documentclass[11pt,reqno]{amsart}

\usepackage{amsmath,amsfonts,pinlabel,amsthm,mathptmx,amssymb,accents}
\usepackage{graphicx}

\catcode`\@=11

\long\def\@savemarbox#1#2{\global\setbox#1\vtop{\hsize\marginparwidth 
  \@parboxrestore\tiny\raggedright #2}}
\newcommand\lref[1]{\ref{#1}%
\@ifundefined{r@DisplaY #1}{}{ (#1)}}

\newcommand\fakelabel[2]{\@bsphack\if@filesw {\let\thepage\relax
   \newcommand\protect{\noexpand\noexpand\noexpand}%
\xdef\@gtempa{\write\@auxout{\string
      \newlabel{#1}{{#2}{\thepage}}}}}\@gtempa
   \if@nobreak \ifvmode\nobreak\fi\fi\fi\@esphack}

\def\SL@margintext#1{{\showlabelsetlabel{\tiny\{\SL@prlabelname{#1}\}}}}
\catcode`\@=12

\def\Empty{}
\newcommand\oplabel[1]{
  \def\OpArg{#1} \ifx \OpArg\Empty {} \else
        \label{#1}
  \fi}
\newtheorem{theoremSt}{Theorem}[section]

\newtheorem{exampleSt}[theoremSt]{Example}
\newtheorem{exerciseSt}[theoremSt]{Exercise}

\newcommand\MakeStEnv[1]{
  \newenvironment{#1}[1]{
  \begin{#1St} \oplabel{##1}%
  \global\def\CrntSt{\thetheoremSt}%
}{ 
  \end{#1St} }
  \newenvironment{#1+}[1]{
  \begin{#1St} \label{##1}%
  \label{DisplaY ##1}%
  \global\def\CrntSt{\thetheoremSt}%
  \def\Labl{##1}\ifx\Labl\Empty{} \else {\em (\Labl)\,}\fi%
}{ 
  \end{#1St} }
}
\MakeStEnv{theorem}
\MakeStEnv{corollary}
\MakeStEnv{proposition}
\MakeStEnv{lemma}
\MakeStEnv{definition}
\MakeStEnv{conjecture}
\MakeStEnv{problem}
\MakeStEnv{question}

\long\def\realfig#1#2{
\begin{figure}[htbp]
\includegraphics{#1}
\caption[#1]{#2}
\oplabel{#1}
\end{figure}}

\newlength{\saveu}

\newenvironment{pf*}[1]{%
 \begin{proof}[#1]%
}{ 
 \end{proof}
}

\newcommand{\finishproof}[1]{ 
  \def\FPArg{#1}
  \ifx\FPArg\Empty
        \newcommand\FPArg{\CrntSt}  \fi
  \smallbreak\noindent\makebox[\textwidth]{\hfill\fbox{\FPArg}}
  \medbreak\noindent
}

\newcommand\AAA{{\mathcal A}}
\newcommand\BB{{\mathcal B}}
\newcommand\CC{{\mathcal C}}

\newcommand\EE{{\mathcal E}}
\newcommand\FF{{\mathcal F}}

\newcommand\HH{{\mathcal H}}

\newcommand\JJ{{\mathcal J}}

\newcommand\LL{{\mathcal L}}
\newcommand\MM{{\mathcal M}}
\newcommand\NN{{\mathcal N}}

\newcommand\PP{{\mathcal P}}

\newcommand\RR{{\mathcal R}}

\newcommand\TT{{\mathcal T}}

\newcommand\CH{{\CC\HH}}

\newcommand\PMF{{\PP\kern-2pt\MM\FF}}

\newcommand\PML{{\PP\kern-2pt\MM\LL}}

\newcommand\ep{\epsilon}

\newcommand\hhat{\widehat}

\newcommand\union{\cup}
\newcommand\intersect{\cap}
\newcommand\bbR{{\mathord{\text{I\kern-2pt R}}}}        
\newcommand\bbH{{\mathord{\text{I\kern-2pt H}}}}        

\newcommand\C{{\mathbb C}}

\newcommand\R{{\mathbb R}}

\newcommand\Hyp{{\mathbb H}}

\newcommand\PSL[1]{\text{PSL}_{#1}}

\newcommand\bigrightarrow[1]{\hbox to #1{\rightarrowfill}}
\newcommand\bigleftarrow[1]{\hbox to #1{\leftarrowfill}}

\newcommand\homeo{\cong}
\newcommand\boundary{\partial}
\newcommand\semidir{\mathrel{\hbox{\vrule depth-.03ex height1.1ex\kern-0.15em$\times$}}}

\newcommand\del{\nabla}

\newcommand{\ssm}{\smallsetminus}

\numberwithin{equation}{section}

\catcode`\@=11

\def\subsection{\@startsection{subsection}{2}%
  \z@{.5\linespacing\@plus.7\linespacing}{.5em}%
  {\normalfont\bfseries\centering}}

\def\section{\@startsection{section}{1}%
  \z@{.7\linespacing\@plus\linespacing}{.5\linespacing}%
  {\normalfont\large\bfseries\centering}}

\def\subsubsection{\@startsection{subsubsection}{3}%
  \z@{.5\linespacing\@plus.7\linespacing}{-.5em}%
  {\normalfont\bfseries}}

\catcode`\@=12

\newcommand{\fsubd}{\mathrel{{\scriptstyle\searrow}\kern-1ex^d\kern0.5ex}}
\newcommand{\bsubd}{\mathrel{{\scriptstyle\swarrow}\kern-1.6ex^d\kern0.8ex}}
\newcommand{\fsubeq}{\mathrel{\raise-.7ex\hbox{$\overset{\searrow}{=}$}}}
\newcommand{\bsubeq}{\mathrel{\raise-.7ex\hbox{$\overset{\swarrow}{=}$}}}

\newcommand{\base}{\operatorname{base}}

\newcommand{\bbar}{\overline}

\newcommand{\tsh}[1]{\left\{\kern-.9ex\left\{#1\right\}\kern-.9ex\right\}}

\newcommand\MT{{\mathbb T}}

\begin{document}

\title{Windows, cores and skinning maps}

\author{Jeffrey F. Brock}
\address{Brown University}
\author{Kenneth W.  Bromberg}
\address{University of Utah}
\author{Richard D.  Canary}
\address{University of Michigan}
\author{Yair N. Minsky}
\address{Yale University}
\date{\today}
\thanks{Brock was partially supported by NSF grant DMS-1207572, Bromberg was partially supported by NSF grants 
DMS-1207873 and DMS-1509171,
Canary  was partially
supported by NSF grant \hbox{DMS -1306992}, and Minsky was partially supported by NSF grant
DMS-1311844}
\begin{abstract}
We give a generalization of Thurston's Bounded Image Theorem for
skinning maps, which applies to pared 3-manifolds with incompressible
boundary that are not necessarily acylindrical. Along the way we study
properties of divergent sequences in the deformation space of such a
manifold, establishing the existence of compact cores satisfying
a certain notion of uniform geometry.
\end{abstract}

\maketitle



\long\def\vfig#1#2#3{
\begin{figure}[htbp]
\includegraphics[height=#1]{#2}
\caption[#2]{#3}
\oplabel{#2}
\end{figure}}

\newcommand\epzero{\ep_0}
\newcommand\epone{\ep_1}
\newcommand\epotal{\ep_{\rm u}}
\newcommand\kotal{k_{\rm u}}
\newcommand\Kmodel{K_0}
\newcommand\Kone{K_1}
\newcommand\Ktwo{K_2}
\newcommand\bdry{\partial} 
\newcommand\stab{\operatorname{stab}}
\newcommand\nslices[2]{#2|_{#1}}
\newcommand\ME{M\kern-4pt E}
\newcommand\bME{\overline{M\kern-4pt E}}
\renewcommand\del{\partial}
\newcommand\s{{\mathbf s}}
\newcommand\pp{{\mathbf p}}
\newcommand\qq{{\mathbf q}}
\newcommand\uu{{\mathbf u}}
\newcommand\vv{{\mathbf v}}
\newcommand\zero{{\mathbf 0}}
\newcommand{\cB}{{\mathcal B}}
\newcommand\mm{\operatorname{\mathbf m}}
\newcommand\intM{\accentset{\circ}{M}}
\newcommand\intP{\accentset{\circ}{P}}
\newcommand\tc[1]{\accentset{\circ}{#1}}
\renewcommand{\CH}{{\mathcal C}\kern-2pt\mathit{H}}

\section{Introduction}

Critical to Thurston's Geometrization Theorem for Haken 3-manifolds
was a fixed-point problem, phrased for a self-mapping of the
deformation space of a hyperbolic 3-manifold with boundary. This {\em
  skinning map} implicitly describes how to enhance a topological gluing
of a 3-manifold along its boundary with geometric information; a fixed
point for the skinning map realizes a geometric solution to the gluing
problem, resulting in a hyperbolic structure on the gluing.

Beyond its utility in geometrization, the map itself reveals more
quantitatively the relationship between topological and geometric
features of a hyperbolic 3-manifold.  Indeed, Thurston's Bounded Image
Theorem (see Thurston \cite{thurston-bangor}, Morgan \cite{Morgan},
Kent \cite{kent-skin}) which provides the desired fixed-point,
guarantees that for an acylindrical 3-manifold, the image of the
skinning map is a bounded subset of Teichm\"uller space.  In this
paper, we investigate the skinning map in the more general case when
the 3-manifold is only assumed to have incompressible boundary.  Here,
the image of the skinning map need not be bounded, but the restriction
onto any essential subsurface of the boundary that is homotopic off of
the characteristic submanifold is bounded. One may think of this as a
strong form of Thurston's Relative Compactness Theorem.  Along the
way, we refine our understanding of the interior geometry of a
hyperbolic 3-manifold, establishing existence of a uniformly
controlled family of compact cores for each deformation space of such
manifolds.

\medskip

Let $M$ be a compact orientable hyperbolizable 3-manifold whose boundary is incompressible
and contains no tori and let $CC_0(M)$ denote the set of convex cocompact
hyperbolic structures on the interior of $M$. 
A convex cocompact hyperbolic structure $N_\rho$ on the interior of $M$ 
gives rise to a holonomy representation $\rho:\pi_1(M)\to {\rm PSL}(2,\mathbb C)$ and induces a well-defined
conformal structure on its boundary $\partial M$.
Bers \cite{bers-survey} shows that one obtains an
identification of $CC_0(M)$ with $\TT(\boundary M)$, the
Teichm\"uller space of conformal structures on $\boundary M$. The skinning map
$$
\sigma_M : CC_0(M) \to \TT(\bbar{\boundary M})
$$
records the asymptotic geometry of the ``inward-pointing'' end of the cover associated to each boundary component.
If $M$ has connected boundary $S$ and $N_\rho$ is in $CC_0(M)$, then the cover $N_S$ of $N_\rho$
associated to $\pi_1(S)$ is
quasifuchsian, i.e. a point in $CC_0(S\times [0,1])$, so  may be identified
to a point $(X,Y)\in\TT(S)\times \TT(\bbar S)$. Then, $\sigma_M(\rho)=Y$ and $X$ is the point
in $\TT(S)$ associated to $\rho$ by the Bers parametrization. (The skinning map will be defined
more carefully, and in greater generality, in Section \ref{deformation}.)

A compact, orientable, hyperbolizable 3-manifold $M$ is said to be {\em acylindrical} if it contains no essential annuli, or,
equivalently, if $\pi_1(M)$ does not admit a non-trivial splitting over a cyclic subgroup.
In this setting, Thurston's Bounded Image Theorem has
the following form.

\medskip\noindent
{\bf Thurston's Bounded Image Theorem:} {\em If $M$ is a compact,
  orientable, acylindrical, hyperbolizable 3-manifold with no torus
  boundary components, 
then the skinning map $\sigma_M : CC_0(M) \to \TT(\bbar{\boundary M})$ has bounded image.}

\medskip

The skinning map has been studied extensively when $M$ is acylindrical.  This study has focussed on
obtaining bounds on the diameter of the skinning map in terms of the topology of $M$ and the geometry
of its unique hyperbolic metric with totally geodesic boundary, see Kent \cite{kent-skin} and Kent-Minsky \cite{kent-minsky}.
In the case that $M$ is not required to be acylindrical, it is known that the skinning map is non-constant
(Dumas-Kent \cite{dumas-kent1}) and that its image is Zariski dense in the  appropriate 
character variety (Dumas-Kent \cite{dumas-kent2}).
Recently, Dumas  \cite{dumas-finite}  has shown that the skinning map is finite-to-one.

In order to state our generalization  of Thurston's Bounded Image Theorem to the setting where
$M$ is only assumed to have incompressible boundary, we recall that
the {\em characteristic submanifold}  $\Sigma(M)$ is a minimal collection of 
solid and thickened tori and  interval bundles in $M$ whose frontier  is a collection
of essential annuli such that every (embedded) essential annulus in  $M$ is isotopic into $\Sigma(M)$
(see Johannson \cite{johannson} or Jaco-Shalen \cite{JS}).
Thurston \cite{thurstonIII} defines the {\em window} of $M$ to be the union of the
interval bundles in $\Sigma(M)$, together with a regular neighborhood of each
component of the frontier of $\Sigma(M)$ which is not homotopic into an interval bundle.
Let $\boundary_{nw}M$ denote the
intersection of $\boundary M$ with the complement of the
window. Components of $\boundary_{nw}M$ are either components of the
intersection of $\boundary M$ with the (relatively) acylindrical pieces of the
decomposition, or annuli in the boundaries of the solid or thickened tori pieces.

Our main theorem asserts that any curve in $\boundary_{nw}M$, equivalently any curve  in $\partial M$ which
may be homotoped off the characteristic submanifold,
has uniformly bounded length in the hyperbolic structures which arise in the skinning image.

\begin{theorem}{main gen}
Let $M$ be a compact, orientable, hyperbolizable 3-manifold
whose boundary is incompressible.
For each curve $\alpha$ in $\boundary_{nw} M$,
its length $\ell_Y(\alpha)$ is bounded as $Y$ varies over the image of $\sigma_M$.
\end{theorem}

If $W$ is an essential subsurface of $\boundary M$, then 
$\sigma_M$ induces a map
$$
\sigma_M^W : CC_0(M) \to \mathcal{F}(W)
$$
where $\mathcal{F}(W)$ is the Fricke space of all hyperbolic
structures on the interior of $W$.
Notice that these
hyperbolic structures are allowed to have either finite or infinite
area. Our main
theorem immediately translates to the fact that $\sigma_M^W$ has bounded image:

\begin{corollary}{maincor gen}
Let $M$ be a compact, orientable, hyperbolizable 3-manifold
whose boundary is incompressible.
If $W$ is a component of  $\boundary_{nw} M$, then
the image of $\sigma_M^W$ is bounded in $\mathcal{F}(W)$.
\end{corollary}

Note that these statements allow $\boundary M$ to have torus components, and
moreover the theorem applies more
generally, when we consider the space $AH_0(M)$ of all hyperbolic structures on the interior of  $M$.
Given an element of $AH_0(M)$ with holonomy
representation $\rho:\pi_1(M)\to PSL(2,\C)$ and quotient manifold $N_\rho$, 
we obtain an end invariant
$\sigma_M(\rho)$ on the non-toroidal part of the boundary,  which
records the asymptotic geometry of the inward-pointing end of the
covers of $N_\rho$ associated
to those boundary components of $M$. This ending invariant consists of a multicurve, known as the parabolic locus,
and either a finite area hyperbolic structure or a filling lamination on each component of the complement.  If a curve 
$\alpha$ on $ \partial M$ is homotopic into a component of the complement of the parabolic locus which has a hyperbolic
structure, then $\ell_{\sigma_M(\rho)}(\alpha)$ denotes the length of the geodesic representative of $\alpha$ in
this hyperbolic structure; if $\alpha$ is homotopic into the parabolic locus then $\ell_{\sigma_M(\rho)}(\alpha)=0$; and
otherwise $\ell_{\sigma_M(\rho)}(\alpha)=+\infty$. 

With these definitions, Theorem \ref{main gen} continues to hold. In
addition, $\sigma_M^W$ can still be defined on $AH_0(M)$ when $W$ is a component of
$\boundary_{nw}M$, and Corollary
\ref{maincor gen} holds as well. Notice that Corollary \ref{maincor gen} contains Thurston's original
Bounded Image Theorem as a special case.
Both results also have natural 
generalizations to the pared setting, which we will state and prove in Section 5.

\medskip

Along the way to proving Theorem \ref{main gen}, we will develop some
tools and structural results on the geometry of hyperbolic 3-manifolds
that may be of independent interest. One is a new tool for studying the
geometry of surface groups, and the other is a uniformity statement
for compact cores of manifolds in a deformation space $AH_0(M)$.

\subsubsection*{Length bounds in surface groups}
In \cite{ELC1,ELC2}, Brock-Canary-Minsky
give an explicit connection between the geometry at infinity of a
surface group $[\rho]\in AH(S)$ and its internal geometry, 
by means of a ``model manifold'' combinatorially built up from the
ending data $\nu(\rho)$. In Section \ref{length bounds} we use the
structure of this model, together with 
some refinements proved in \cite{BBCM}, to ``reverse engineer'' a usable criterion
for bounding the length at infinity of a curve in $S$, based on the
situation in the interior.

Specifically, let $\alpha$ be a simple curve in $S$, and $\hhat \alpha$
a representative of $\alpha$ in the the \hbox{3-manifold} $N_\rho$, which is
contained in a ``level surface'' in a product structure on $N_\rho$.
Let  $\mathcal C(\hhat\alpha,L)$ denote the set of
curves on $S$ which intersect 
$\alpha$ essentially and whose geodesic representatives in
$N_\rho$  have length at most $L$ and do not ``lie above''
$\hhat\alpha$ (in a sense to be made precise in Section
\ref{hierarchies etc}). We will show
that, given a length bound on $\hhat\alpha$ and constraints on
$\mathcal C(\hhat\alpha,L)$, we can obtain bounds on the length of $\alpha$ in
the ``bottom'' conformal structure $\nu^-(\rho)$.

The idea here is that, in order for $\alpha$ to be considerably longer
in $\nu^-(\rho)$ 
than $\hhat\alpha$ is, there must be some kind of geometric complexity
between $\hhat\alpha$ and the bottom end of $N_\rho$, and this is what
$\CC(\hhat\alpha,L)$ captures. We give a loose statement of the theorem
here, in the closed case; see section \ref{length bounds} for a more
carefully quantified general version.

\medskip\par\noindent
    {\bf Theorem \ref{upper bound}.}
{\em
    Let $S$ be a closed surface and $\alpha$ an essential simple
    closed curve in $S$. Let $\hhat\alpha$ be a representative of
    $\alpha$ in $N_\rho$ for $[\rho]\in AH(S)$ which is contained in a level surface that
    avoids the thin part of $N_\rho$. Given a length bound on
    $\hhat\alpha$, an upper bound on the  number of elements of $\mathcal C(\hhat\alpha,L)$
    (for suitable $L$) and a lower bound on the length of all elements of
    $\mathcal C(\hhat\alpha,L)$, we obtain an upper bound on
    $\ell_{\nu^-(\rho)}(\alpha)$. 
}

\subsubsection*{Uniform Core Models}
If $M$ is acylindrical, then it follows from Thurston's Compactness Theorem \cite{thurston1} and
work of Anderson-Canary \cite{AC-cores} 
that there is a fixed metric on $M$ such that
for each  $[\rho]\in AH_0(M)$ there
is a uniformly bilipschitz embedding of $M$ into
$N_\rho$ in the homotopy class of $\rho$.
The image of this embedding is a compact core with uniformly bounded geometry.
In our setting this is no longer possible, since a sequence of hyperbolic manifolds in $AH_0(M)$
may ``pull apart'' along Margulis tubes. Moreover, one must take into account the action of the
outer automorphism group, which will be infinite if $M$ is not acylindrical.

In Section \ref{uniform cores section} we define a notion of a
{\em model core} for an element  $[\rho]\in AH_0(M)$. Roughly, a model core is a metric $m$
on the complement $M_{\mathcal A}$ of a collection of solid and thickened tori
in $M$.
We say that a model core 
{\em  controls} $N_\rho$ if there is an embedding of $M$ into $N_\rho$
which is 2-bilipschitz with respect to $m$ on $M_\mathcal A$, 
takes the components of $M\ssm M_\AAA$ into the thin part of
$N_\rho$, and lies
in the homotopy class of $\rho\circ\phi_*$ where $\phi:M\to M$ is a homeomorphism which
is the identity on the complement of the characteristic submanifold.
The image of the embedding is a compact core for $N_\rho$ with uniformly
bounded geometry on $M_\mathcal A$ (and, in particular, on the complement of the characteristic submanifold).

\medskip\noindent
{\bf Theorem \ref{uniform cores}.}
{\em
If $M$ is a hyperbolizable 3-manifold with incompressible boundary
then there is a finite collection of model cores so that
each element of $AH_0(M)$ is controlled by one of them. 
    }

\medskip
    
The proof of Theorem \ref{uniform cores} utilizes a version of
Thurston's Relative Compactness Theorem \cite{thurstonIII} from Canary-Minsky-Taylor \cite{CMT}
and the analysis of the relationship between algebraic and geometric limits of sequences of isomorphic Kleinian groups
carried out by Anderson, Canary, and McCullough \cite{AC-cores,ACM}.

\medskip\noindent
{\bf Summary of Proof of Theorem \ref{main gen}:} It is instructive to first think about our
argument in the case where $M$ is 
acylindrical, which is the setting of Thurston's original Bounded Image Theorem. We will
also simplify the discussion by assuming that $\partial M=S$ is
connected. Let $\alpha$ be a fixed closed curve in $S$. 

Theorem \ref{uniform cores} gives us, in this case, 
a finite collection $\{C_1,\ldots, C_r\}$ of compact Riemannian manifolds with boundary
so that any $[\rho]\in AH(M)$ has a compact core which is 2-bilipschitz to some $C_i$.
(In the acylindrical case this can be directly obtained from
Thurston's Compactness Theorem \cite{thurston1} and the work of 
Anderson-Canary \cite{AC-cores} on cores and limits). 
For such a compact core $C$, its boundary $F$ lifts to a level surface
$\hhat F$ for the cover $N_S$ associated to $\rho(\pi_1(S))$,
which contains a representative $\hhat\alpha$ of $\alpha$ whose length
is uniformly bounded.
Let $\hhat C$ be the component of the pre-image of $C$ in $N_\rho^S$
which contains $\hhat F$ in its boundary. Since $M$ is acylindrical, each component of $N_S\ssm\hhat C$ which
lies below $\hhat\alpha$ is simply connected. 
Therefore, every bounded length closed geodesic which
is not above $\hhat\alpha$ must intersect $\hhat C$. Moreover, again since $M$ is acylindrical, distinct closed
geodesics in $N_S$ project to distinct geodesics in $N_\rho$. Since $C$ has uniformly bounded geometry
we obtain an upper bound on the number of such geodesics and a lower
bound on their length. Theorem \ref{upper bound} then implies
that $\ell_{\nu^-(\rho^S)}(\alpha),$ which is exactly $\alpha$'s
length in the skinning structure, is uniformly bounded.

In order to generalize this proof to the setting where $M$ is only assumed to have incompressible boundary, 
we need the more general statement of Theorem \ref{uniform cores} about model cores. 
This theorem gives us, for $[\rho]\in AH_0(M)$,  a compact core $C$
in $N_\rho$ whose geometry is uniformly controlled on the ``non-window'' part
(as well as on parts of the window). Picking a component
$F$ of $\boundary C$ and a curve $\alpha$ on the non-window part of
$F$, we can again lift to the corresponding cover $N_\rho^F$. Now we
learn that a bounded length geodesic lying below $\hhat F$ will either
intersect $\hhat C$ or lie in a complementary region of $\hhat C$ which
retracts to the window part. In the latter case this geodesic cannot
represent a curve that intersects $\alpha$ essentially, and hence does
not count as a member of $\mathcal C(\hhat\alpha,L)$. An application of
Theorem \ref{upper bound}, with a bit of additional care, 
again gives us the desired bound. 

\medskip\noindent
{\bf Relationship to proofs of Thurston's Bounded Image Theorem:}
Our proof is new even in the acylindrical setting. The only known proof
of Thurston's original Bounded Image Theorem, due to Kent \cite{kent-skin},
uses the work of Anderson-Canary \cite{AC-cores}  to extend the skinning map
continuously to all of $AH(M)$ and then applies Thurston's Compactness Theorem, which
gives that $AH(M)$ is compact,
to conclude that the image is bounded. In our more general setting, the skinning map
does not extend continuously to $AH(M)$, in any reasonable sense, and $AH(M)$ fails to
be compact, so it was necessary to develop a new approach to the proof.

\subsubsection*{Potential applications:}
The  model manifolds constructed in \cite{ELC1,ELC2} are bilipschitz
equivalent to $N_\rho$ for any $[\rho]\in AH(M)$, but the bilipschitz
constant is not in general uniform over all of $AH(M)$. For the case
of surface groups, we do have uniformity once the genus is fixed, but
a uniform model construction that would apply in general would be
helpful for a better global understanding of 
$AH(M)$.

When $M$ is acylindrical, 
Thurston's Bounded Image Theorem suggests a natural candidate for this
uniform model: 
In this case, there is a unique hyperbolic \hbox{3-manifold} $C_M$
with totally geodesic boundary homeomorphic to $M$. Let \hbox{$Y\in\TT(\boundary M)$}
denote the conformal structure of this boundary. 
A model for the
hyperbolic \hbox{3-manifold} in $CC_0(M)$ with 
conformal structure $X\in\mathcal T(\partial M)$ can then be assembled
from $C_M$ and the quasifuchsian hyperbolic 3-manifold with conformal
structure $(X,Y)$. One expects Thurston's Bounded Image Theorem 
and the Bilipschitz Model Manifold Theorem from \cite{ELC2} to be the main
tools for showing that this model is uniform over $CC_0(M)$.

Our relative bounded image theorem should play a similar role in constructing
uniform models for hyperbolic 3-manifolds with finitely generated,
freely indecomposable fundamental group. The complement of the
characteristic submanifold admits a cusped hyperbolic metric with totally
geodesic boundary, and this should play the role of the manifold
$C_M$.

\medskip

The relative bounded image theorem should also be a tool in the study
of the local topology of $AH(M)$.  We recall that Bromberg
\cite{bromberg-pt} and Magid \cite{magid} showed that 
$AH(S)$ fails to be locally connected.  In earlier work \cite{BBCM0}, 
we used Thurston's Bounded Image Theorem to show that
if $M$ is acylindrical and \hbox{$[\rho]\in AH(M)$} is quasiconformally
rigid (i.e. every component of the conformal boundary of $N_\rho$
is a thrice-punctured sphere), then $AH(M)$ is locally connected
at $\rho$. We hope our relative bounded image theorem will allow
us to attack the following conjecture.

\medskip\noindent
{\bf Conjecture:} {\em Let $M$ be a compact, orientable, hyperbolizable 3-manifold whose 
boundary is incompressible. If $[\rho]\in AH(M)$ is quasiconformally rigid, then
$AH(M)$ is locally connected at $\rho$.}

\medskip

The relative bounded image theorem may also be useful  in establishing a particularly mysterious
step in Thurston's original proof of his Geometrization Theorem (see Morgan \cite[Section 10]{Morgan}).
If $\tau:M\to M$ is an orientation-reversing homeomorphism so that the 3-manifold $M_\tau$ is
atoroidal, then $\tau$ induces a map $\tau_\sharp:\mathcal T(\partial M)\to \mathcal T(\overline{\partial M})$.
(Here we allow $M$ to have two components and make the natural adjustments to the notation.)
Thurston claimed the following result, which, along with  Maskit's combination theorems \cite{maskit}, completed the final
step of the proof of his Geometrization Theorem.

\medskip\noindent
{\bf Theorem:} (Thurston) {\em If $M_\tau$ is atoroidal, then there exists $n\in\mathbb N$ so that
$(\tau_\sharp\circ\sigma_M)^n$ has bounded image.}

\medskip

As far as we know, no one currently knows a proof of Thurston's result. All written versions of Thurston's
proof of hyperbolization only produce a fixed point of $\tau_\sharp\circ\sigma_M$ 
(which suffices to hyperbolize $M_\tau$), see for example Kapovich \cite{kapovich-book}, McMullen \cite{mcmullen-geom}, 
Morgan \cite{Morgan} and Otal \cite{otal-hyp}.

\section{Background}

In this section we recall some standard ideas and notation, with a few
new notions particular to this paper. 
Section \ref{deformation} is a brief review of deformation spaces, end
invariants and the definition of skinning maps.
Section \ref{JSJ section} reviews pared manifolds and the Jaco-Shalen Johannson theory of
characteristic submanifolds, and introduces the notion of {\em robust
  systems of annuli} which give useful ways to cut up 3-manifolds that
correspond to the way in which sequences of hyperbolic structures
diverge along cusps.
Section \ref{partial convergence} discusses algebraic and geometric
convergence for representations, and introduces language for
describing divergent sequences that converge on subgroups.
Section \ref{hierarchies etc} recalls the notion of markings and
hierarchies for curve complexes of surfaces, which are the
combinatorial ingredients for the model manifold construction of
\cite{ELC1} and are needed here in the proof of Theorem \ref{upper
  bound}. 
Section \ref{relative} shows how to use Thurston's Relative
Compactness Theorem to extract convergent behavior from a diverging
sequence of elements of $AH_0(M,P)$. 
Corollary \ref{ourrelcompactness} is the
main result we will need, adapting ideas from Canary-Minsky-Taylor
\cite{CMT} to
provide a robust system of annuli for such
a sequence so that a subsequence  converges in a suitable sense on
its complementary pieces.

\subsection{Deformation spaces and skinning maps}
\label{deformation}
We begin with a quick review of the terminology for deformation spaces
of hyperbolic 3-manifolds.

\subsubsection*{Hyperbolic 3-manifolds}
A hyperbolic 3-manifold is the quotient $N=\mathbb H^3/\Gamma$ of  hyperbolic 3-space by a group
$\Gamma$ of isometries acting freely and properly discontinuously. We will always assume that
$N$ is orientable, so 
$$\Gamma\subset{\rm Isom}_+(\mathbb H^3)\cong{\rm PSL}(2,\mathbb C).$$
The {\em convex core} $C(N)$ of $N$ is the smallest convex submanifold of $N$.

There exists a constant $\mu_3>0$, called the {\em Margulis constant}, such that if $\epsilon<\mu_3$ and 
$$N_{(0,\epsilon)}=\{x\in N\ |\ {\rm inj}_N(x)<\epsilon\}$$
is the set of points in $N$ with injectivity radius less than $\epsilon$, then
every component of $N_{(0,\epsilon)}$ is a {\em Margulis region}, i.e.  either a {\em Margulis tube}, 
(a metric neighborhood of a closed geodesic) or a {\em cusp region} (a quotient
of a horoball by a parabolic subgroup). If $\gamma$ is a nontrivial element of $\pi_1(N)$, we denote by 
$\MT_\ep(\gamma)$ the component of $N_{(0,\epsilon)}$, possibly empty, associated to the maximal abelian
subgroup containing $\gamma$, and similarly define $\MT_\ep(a)$ (or $\MT_\ep(A)$)  if $a$ is a
curve in $N$ (or $A$ is an incompressible annulus in $N$). 
We fix an explicit value 
$\ep_0\in (0,\mu_3)$ which we use as a default value for Margulis regions.
(For convenience, we choose the value $\ep_0$ used in the Bilipschitz Model Theorem from \cite{ELC2}).
In particular, we will use the shorthand $\MT(a)=\MT_{\ep_0}(a)$.

Let $N^0$ be obtained from $N$ by removing all cusp regions in $N_{(0,\ep_0)}$.
If $N$ has finitely generated fundamental group,
then $N^0$ contains a {\em relative compact core} $(M,P)\subset (N^0,\partial N^0)$,
where $M$ is a codimension zero submanifold such that the inclusion into $N$ is a homotopy equivalence,
each component $X$ of $\partial N^0$
contains exactly one component  $Q$ of $P$,  and the inclusion  of $Q$ into $X$ is a homotopy
equivalence (see Kulkarni-Shalen \cite{kulkarni-shalen} or McCullough \cite{mcculloughRCC}).
More generally, if $\ep\in(0,\mu_3)$, we define $N^\ep$ to be the complement of the
non-compact components of $N_{(0,\ep)}$ and note that one may similarly define relative compact cores for
$N^\ep$.

\subsubsection*{Pared manifolds}
A {\em pared manifold} is a pair $(M,P)$ where $M$ is a compact oriented
irreducible 3-manifold and $P$ is a closed subsurface of $\boundary M$ whose
components are incompressible annuli and tori, such that
\begin{enumerate}
\item 
every abelian, non-cyclic subgroup of $\pi_1(M)$ is conjugate into a subgroup of 
$\pi_1(Q)$ where $Q$ is a component of $P$, and
\item 
any $\pi_1$-injective map $f:(S^1\times [0,1],S^1\times \{0,1\})\to (M,P)$ of an annulus into $M$ is homotopic,
as a map of pairs, to a map with image in $P$. 
\end{enumerate}

Thurston's Geometrization Theorem, see Morgan \cite[Theorem ${\rm B}'$]{Morgan}, implies that $(M,P)$ is pared if and only if
it is the relative compact core of a hyperbolic \hbox{3-manifold} with finitely generated fundamental group.

The closure $\overline{\partial M-P}$ of the complement of $P$ in $\partial M$ is called the
{\em pared boundary} and denoted $\boundary_0(M,P)$. We say that $(M,P)$ has {\em pared incompressible
boundary} if each component of $\partial_0(M,P)$ is incompressible.

Bonahon's Tameness Theorem \cite{bonahon} implies that if $N^0$ has finitely
generated fundamental group and  its relative compact core $(M,P)$ has pared incompressible boundary,
then $N^0-{\rm int}(M)$ is homeomorphic to $\partial_0(M,P)\times [0,\infty)$. It follows that the relative compact core is 
well-defined up to isotopy in this case.

\subsubsection*{Spaces of representations}
If $G$ is a group let $D(G)$ denote the subset of ${\rm Hom}(G,\PSL 2(\mathbb C))$ consisting of discrete, faithful
representations.
For a 3-manifold $M$, let $D(M) = D(\pi_1(M))$, and for a pared manifold $(M,P)$,
let $D(M,P)$ denote the subset of $D(M)$ consisting of representations such that
$\pi_1(Q)$ maps to a parabolic subgroup for each component $Q$ of $P$.
Let $AH(M,P)$ be the set of conjugacy classes in $D(M,P)$. If $P$ is empty or consists only of tori, then
we  use $AH(M)$ as shorthand for $AH(M,P)$.
We give ${\rm Hom}(G,\PSL 2(\mathbb C))$ the compact-open topology, $D(M)$ and $D(M,P)$ inherit
the subspace topology and we give $AH(M,P)$ the quotient topology.

If $\rho\in D(M,P)$, let
$N_\rho = \Hyp^3/\rho(\pi_1(M))$ be the quotient manifold, and let $(M_\rho,P_\rho)$ be a
relative compact core for $N_\rho^0$.
Note that the conjugacy class $[\rho]\in AH(M)$ determines a homotopy class
of homotopy equivalences $h_\rho:M\to M_\rho$ and associated maps of pairs
$h_\rho:(M,P)\to (M_\rho,P_\rho)$. Let $AH_0(M,P)$ comprise those
elements of $AH(M,P)$ for which the homotopy class of $h_\rho$ contains a 
an orientation-preserving pared homeomorphism.

Let $GF(M,P) \subset AH(M,P)$ denote the geometrically finite representations
(i.e. those for which the intersection of the convex core $C(N_\rho)$ with $N_\rho^0$ is compact)
whose {\em only} parabolic subgroups correspond to components of $P$. 
(In this case, $(C(N_\rho),C(N_\rho)\cap \partial N^0_\rho)$ is a relative compact core unless
$\rho$ is Fuchsian, in which case $C(N_\rho)$ is two-dimensional.)
Let $GF_0(M,P) \subset GF(M,P)$ denote the geometrically finite
elements which also lie in $AH_0(M,P)$. If $P$ is empty, then $GF_0(M,P)$ agrees with the
deformation space $CC_0(M)$ introduced in the introduction.

\subsubsection*{End invariants}
If $(M,P)$ has pared incompressible boundary,
the Bers parameterization (see Bers \cite{bers-spaces,bers-survey} or \cite[Chapter 7]{canary-mccullough})
gives rise to a homeomorphism
$$
B: GF_0(M,P) \to \TT(\boundary_0(M,P)).
$$
Here $\TT(S)$, for a compact surface $S$, denotes the Teichm\"uller
space of marked, finite area hyperbolic structures on the interior of $S$.
A special case of this is the product case $M = S\times [0,1]$ for a
compact surface $S$, and $P = \boundary S \times [0,1]$. In this case
$GF_0(S\times [0,1],\partial S\times [0,1]) \homeo \TT(S)\times\TT(\bar S)$ denotes the space of
{\em quasifuchsian representations} of $S$, where $\bar S$ is $S$ with
the opposite orientation. We use the notation $QF(S)$ for
$GF_0(S\times [0,1],\partial S\times [0,1])$, $D(S)$ for $D(S\times
[0,1],\partial S\times [0,1])$, and
$AH(S)$ for $AH(S\times [0,1],\partial S\times [0,1])$.

Thurston's {\em  end invariants} generalize Bers' parameters. An end invariant for
a point $[\rho]$ in
$AH_0(M,P)$ is the following data: A multicurve
$p$ in $\boundary_0(M,P)$, called the parabolic locus, and for each
component $X$ of $\boundary_0(M,P) \setminus p$ that is not a 3-holed
sphere, either a point
of $\TT(X)$, in which case $X$ is called a geometrically finite component, or a  filling geodesic lamination on $X$,
in which case $X$ is called a geometrically infinite component. 
The parabolic locus $p$ is a collection of core curves of components
of $h_\rho^{-1}(P_\rho)$. Each geometrically finite component
corresponds to a component of the conformal boundary at infinity which
bounds a component of the complement of the relative compact core of
$(N_\rho)^0$, and each geometrically infinite component corresponds to
a component which is ``simply degenerate'' in the sense of Thurston,
in which case the lamination is the ending lamination of that
component. See \cite{ELC1,ELC2} for a detailed discussion.

We write the end invariant of $\rho\in AH_0(M,P)$ as $\nu(\rho)$. If $S$ is a component of
$\boundary_0(M,P)$ and $\alpha$ is a curve on $S$, then we define 
$\ell_{\nu(\rho)}(\gamma)\in [0,+\infty]$ to
be $0$ if $\gamma$ is homotopic into the parabolic locus, to be the length of $\gamma$ in
the associated hyperbolic metric if $\gamma$ is contained in a geometrically finite component
of the parabolic locus, and to be $+\infty$ otherwise.

In the quasifuchsian space $QF(S)$ we write $\nu(\rho)$  as a pair $(\nu^+(\rho),\nu^-(\rho))$
associated to the two copies of the surface.  The ordering of this
pair is determined by keeping track of the orientation of $S$ and of
the quotient 3-manifold. In particular given an orientation-preserving
embedding $S\times[0,1]$ as a relative compact core and identifying
$N_\rho\setminus \MT(\boundary S)$ with $S\times\R$, $\nu^+$
corresponds to the component $S\times (1,\infty)$, and $\nu^-$ to the
$S\times(-\infty,0)$, which we call the {\em top} and {\em bottom ends}, respectively.

Let $\EE(S)$ denote the set of all possible end invariants on a surface $S$. 
The Ending Lamination Theorem \cite{ELC2} tells us that the map
$$
\nu_{(M,P)}: AH_0(M,P) \to \EE(\boundary_0(M,P))
$$
is an injection.  In the special case of a compact surface $S$ we obtain
$$\nu_S: AH_0(S) \to \EE(S)\times \EE(\bar S)$$
which has the form $(\nu_S^+,\nu_S^-)$.
We note that, while the Bers parameterization is
a holomorphic isomorphism, there is no natural topology on
$\EE(\boundary_0(M,P))$ for which $\nu$ is even continuous
(see Brock \cite{brock-EIT}).

\subsubsection*{Skinning maps}

If $(M,P)$ has pared incompressible boundary, then 
for each component $S$ of $\boundary_0(M,P)$, we have a restriction
map $D(M,P) \to D(S)$ which induces a map
$$
r_{(M,P)}^S: AH_0(M,P) \to AH_0(S).
$$

Thurston proved that every cover of a geometrically finite Kleinian group
of  infinite volume which has finitely generated fundamental group is itself
geometrically finite (see Morgan \cite[Prop. 7.1]{Morgan}).
If $P$ is empty, or if no non-peripheral curve in $S$ is homotopic into $P$, then
$r_{(M,P)}^S(GF_0(M,P))\subset QF(S)$.
In these cases, using the Bers parametrization, we obtain a map
$$
\hat r_{(M,P)}^S : \TT(\boundary_0(M,P)) \to \TT(S) \times \TT(\bar S)
$$
and, putting this together over all components we get
$$
\hat r_{(M,P)}:
\TT(\boundary_0(M,P)) \to
\TT(\boundary_0(M,P)) \times
\TT(\overline{\boundary_0(M,P)})
$$
which has the form $\hat r_{(M,P)}=(Id,\sigma_{(M,P)})$,
where 
$$
\sigma_{(M,P)} : \TT(\boundary_0(M,P)) \to \TT(\overline{\boundary_0(M,P)}).
$$
is Thurston's skinning map.

To define a skinning map in general, for each component $S$ of $\partial_0(M,P)$
we consider the map
$$\sigma_{(M,P)}^S=\nu_S^-\circ r_{(M,P)}^S:AH_0(M,P) \to
\EE(\bar S).$$
The skinning map is the product
$$\sigma_{(M,P)}=\prod_{S\subset\partial_0(M,P)} \sigma_{(M,P)}^S$$
where the product is taken over all components of $\partial_0(M,P)$.
Notice that 
$$\prod_{S\subset\partial_0(M,P)}  \nu_S\circ r_{(M,P)}^S:
AH_0(M,P) \to\EE(\boundary_0(M,P))\times\EE(\overline{\boundary_0(M,P)}) $$
has the form $(\nu_{(M,P)},\sigma_{(M,P)})$.

\medskip

\subsubsection*{Fricke spaces}
If $S$ is a compact, orientable surface with negative Euler characteristic,
one may consider the {\em Fricke space} $\mathcal F(S)$ of all complete marked
hyperbolic structures on the interior of $S$. Notice that the Teichm\"uller space $\mathcal T(S)$
of all finite area complete hyperbolic structures on $S$ is naturally the boundary of the Fricke space.
If $S$ has genus $g$ and $n$ boundary components, then 
$$\mathcal F(S)\cong \mathbb R^{6g-6+2n}\times[0,\infty)^n$$
and one may define Fenchel-Nielsen coordinates just as for $\mathcal T(S)$, the only difference
being that one has an additional coordinate which records the lengths of the geodesics in the homotopy
class of $\partial S$, see Bers-Gardiner \cite{bers-gardiner} for details.

For an annulus $A$ we declare $\FF(A)$ to be $[0,\infty)$, where the
parameter denotes the translation length of the core of $A$ in a complete
hyperbolic structure on the interior of $A$.

\subsection{Characteristic submanifold theory}
\label{JSJ section}
Jaco-Shalen \cite{JS} and Johannson \cite{johannson} introduced the characteristic submanifold $\Sigma(M)$
of a compact, orientable, irreducible 3-manifold $M$ with incompressible boundary, which is 
a canonical collection of Seifert-fibered and interval bundle submanifolds which contain
all the incompressible annuli and tori in $M$. Johannson \cite{johannson} further proved that every
homotopy equivalence of $M$ can be homotoped to preserve $\Sigma(M)$ and its complement $M-\Sigma(M)$.
In this section, we recall the key properties of the characteristic submanifold in the setting of pared manifolds.
(For a more detailed discussion of the characteristic submanifold in the pared setting, see either 
Morgan  \cite[Sec. 11]{Morgan} or Canary-McCullough \cite[Chap. 5]{canary-mccullough}).) 

We must introduce some additional notation to handle the pared case.
If $B$ is a compact surface with boundary, we say that $f:B\to M$ is {\em admissible in $(M,P)$}
if whenever $C$ is a component of $\partial B$, then either $f(C)\subset {\rm int}(P)$ or
\hbox{$f(C)\subset \partial_0(M,P)$}. 
We say that two admissible maps (or admissible embeddings) $f:B\to M$ and $g:B\to M$ 
are {\em admissibly homotopic} (or {\em admissibly isotopic}) if there exists a homotopy (isotopy)
$H:B\times [0,1]\to M$ from $f$ to $g$ such that if $C$ is a component of $\partial B$, then
either $H(C\times [0,1])\subset P$ or $H(C\times [0,1])\subset \partial_0(M,P)$.
We recall that if $A$ is an annulus or M\"obius band, then an admissible map $f:A \to M$ is an 
{\em essential immersed annulus or M\"obius band in $(M,P)$},
if  $f_*:\pi_1(A)\to\pi_1(M)$ is injective, and $f$ is not admissibly homotopic to a map with image in $\partial M$.
If $f$ is an embedding, we will simply call $f$ an {\em essential annulus or M\"obius band in $(M,P)$.} In this case we will
often abuse notation and also refer to the image $f(A)$ as an essential annulus or M\"obius band in $(M,P)$.
Notice that with this definition annuli which are parallel to components of $P$ are inessential.

The following result encodes the basic properties of the characteristic submanifold
of a pared manifold with pared incompressible boundary.

\begin{theorem}{charprop}
{\rm (Jaco-Shalen \cite{JS}, Johannson \cite{johannson})}
Let $(M,P)$ be a pared $3$-manifold with incompressible pared boundary. 
There exists a codimension zero submanifold $\Sigma(M,P)$
satisfying the following properties: 
\begin{enumerate}
\item Each component $\Sigma_i$ of $\Sigma(M,P)$ is
either 
\subitem(i)
an interval bundle  over a compact surface $F_i$ with negative
Euler characteristic so that $\Sigma_i\cap \partial_0 (M,P)$
is its associated $\partial I$-bundle,
\subitem(ii) a solid torus $V$, or 
\subitem(iii)  a thickened torus $X$ such that $\partial X$  contains a toroidal component of $P$.
\item The frontier \hbox{${\rm Fr}(\Sigma(M,P))$}  of $\Sigma(M,P)$ is a collection of essential annuli
with boundary in $\partial M-P$.
\item If an annular component of $P$ is homotopic into a solid torus component $V$, of $\Sigma(M,P)$, then it
is contained in the interior of a component of \hbox{$\partial V\cap\partial M$}.
\item Any immersed essential annulus or M\"obius band in $(M,P)$ is admissibly homotopic into  $\Sigma(M,P)$.
If the annulus or M\"obius  band is embedded, then it is is admissibly isotopic into $\Sigma(M,P)$.
\item
If $Y$ is a component of \hbox{$M-\Sigma(M,P)$},
then either $\pi_1(Y)$ is non-abelian
or \hbox{$(\overline{Y}, Fr(Y))\cong (S^1\times [0,1]\times [0,1], S^1\times [0,1]\times \{0,1\}) $}
and $Y$ lies between an interval bundle component of
\hbox{$\Sigma(M,P)$} and a thickened or solid torus component of \hbox{$\Sigma(M,P)$}.

Moreover, 
the component of $\Sigma(M,P)\cup Y$ which contains $Y$  is not a submanifold of type (1).
\end{enumerate}
\noindent A submanifold with these properties is unique up to isotopy, and is called the 
{\em characteristic submanifold} of $M$.
\end{theorem}

\realfig{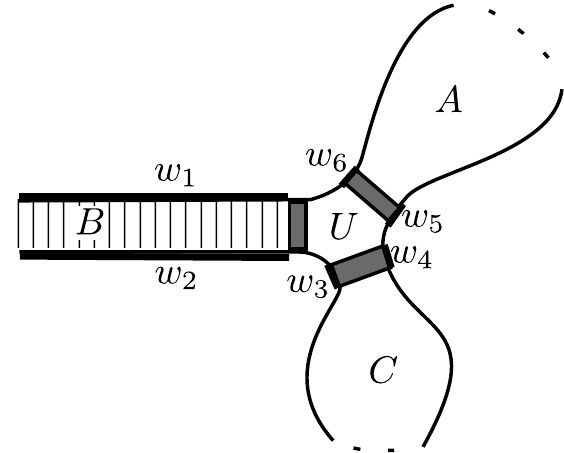}{Schematic of a typical JSJ
 decomposition. $A$ and $C$ represent relatively
  acylindrical pieces. $\Sigma(M,P)$ consists of an $I$-bundle $B$ and a
  solid torus $U$. The window
  surface has six components: $w_1,w_2$ 
  in   $\boundary_0 B$, and four annuli $w_3$--$w_6$.}

\noindent
{\bf Remark:}
For statements of Jaco, Shalen and Johannson's results in a form quite similar to the above statement,
see Canary-McCullough \cite[Thm. 2.9.1 and 5.3.4]{canary-mccullough}. In Johannson \cite{johannson}
(and in Canary-McCullough \cite{canary-mccullough}), every toroidal component of $P$ is contained
in a component of the characteristic submanifold. In this paper, we use the convention of Jaco-Shalen \cite{JS},
that a toroidal component  $T$ of $P$ is only contained in a component of $\Sigma(M,P)$ if there is an essential
annulus with one boundary component in $T$.

\medskip

Thurston defines the {\em window} of $(M,P)$, denoted ${\rm window}(M,P)$, to be the union of the interval bundle components
of $\Sigma(M,P)$ with a regular neighborhood of every component of ${\rm Fr}(\Sigma(M,P))$ which is not
homotopic into an interval bundle component.
The window is an interval bundle over a surface ${\rm wb}(M,P)$, called the {\em window base}, which contains
a copy of the base surface of each interval bundle component of $\Sigma(M,P)$ and an annulus
for each component of ${\rm Fr}(\Sigma(M,P))$ which is not
homotopic into an interval bundle component. 
The {\em window surface} ${\rm ws}(M,P)={\rm window}(M,P)\cap\partial_0(M,P)$
is a two-fold cover of the window base.
We then define 
$$\partial_{nw}(M,P)=\partial_0(M,P)-{\rm ws}(M,P)$$
to be the non-window portion of the boundary. Each component of $\partial_{nw}(M,P)$ is either an annulus
homotopic to a component of ${\rm Fr}(\Sigma(M,P))\cap\partial M$ or a subsurface with negative Euler
characteristic contained in a component of $\partial M-\Sigma(M,P)$.

\subsubsection*{Robust systems of annuli} 
We will see, in Section \ref{relative}, that one may remark a sequence in $AH_0(M,P)$ so that it converges
on the complement of a collection of essential annuli, in a sense to be defined in Section \ref{partial convergence}.
It will be convenient to always choose this collection of essential
annuli so that each solid or thickened torus of $(M,P)$ is either cut
out cleanly in one piece or does not participate at all in the collection.
This is part of a condition we call {\em robustness}, which we
define and study below. 

We say that a collection $\mathcal A$ of essential annuli in $\Sigma(M,P)$ is {\em robust} if
\begin{enumerate}
\item  
the components of $\mathcal A$ are disjoint and mutually non-parallel,
\item
$\mathcal A$ contains the frontier of any thickened torus component $X$ of $\Sigma(M,P)$,
and contains no other annuli homotopic into $X$, and
\item 
if $\mathcal A$ has a component homotopic into a solid  torus component  $V$ of $\Sigma(M,P)$, then
$\mathcal A$ contains the frontier of $V$ and no other annuli homotopic into $V$. Moreover, if $V$ contains
a component of $P$, then the frontier of $V$ is contained in $\mathcal A$.
\end{enumerate}

Any component of $\mathcal A$ which is not homotopic into a solid or thickened torus
component of $\Sigma(M,P)$, must be isotopic into an interval bundle component $\Sigma_i$  of $\Sigma(M,P)$ and hence to
the interval bundle over a two-sided curve in $F_i$ (see \cite[Thm. 2.7.2]{canary-mccullough} for example).
Hence, one may always isotope $\mathcal A$ so that it is the $I$-bundle associated to a 
two-sided multi-curve $a$ in the window base.

A robust collection of essential annuli $\mathcal A$ in $(M,P)$ naturally decomposes $(M,P)$ into 
a collection of solid and thickened tori
and submanifolds which inherit the structure of a pared manifold.
Let $\mathcal N(\mathcal A)$ be a (pared) regular neighborhood of the union of the components of
$\mathcal A$. A solid or thickened torus component of \hbox{$M-\mathcal N(\mathcal A)$} is either a regular
neighborhood of a solid or thickened torus component of $\Sigma (M,P)$ or a regular neighborhood of
an essential M\"obius band in an interval bundle component of $\Sigma(M,P)$.
Let $M_{\mathcal A}$ be obtained from $M-\mathcal N(\mathcal A)$ by deleting all components
which are solid or thickened tori and let
$$P_\mathcal{A}=(M_\mathcal A\cap P)\cup {\rm Fr}(M_\mathcal A).$$

\begin{lemma}{pieces are pared} 
If $(M,P)$ is a pared manifold with pared incompressible boundary and
$\mathcal A$ is robust collection of essential annuli in $(M,P)$,
then each component of $(M_\mathcal A,P_\mathcal A)$ is a pared manifold with
pared incompressible boundary.
\end{lemma}

\begin{proof}
Let $(M',P')$ be a component of $(M_\mathcal A,P_\mathcal A)$.
Every component of $P'$ is either a component of $P$ or isotopic to a component of $\mathcal A$,
hence  is either an incompressible annulus or torus. 

Any non-cyclic abelian subgroup $H$ of $\pi_1(M')$ is also a non-cyclic abelian
subgroup of $\pi_1(M)$, since $\pi_1(M')$ injects into $\pi_1(M)$. Since $(M,P)$ is pared,
$H$ is conjugate in $\pi_1(M)$ to a subgroup of $\pi_1(T)$ where $T$ is a toroidal component of $P$.
Since the frontier of $M'$ is a collection of essential annuli, $T$ must be a toroidal component of
$\partial M'$ and   $H$ 
is conjugate in $\pi_1(M')$ to a subgroup of $\pi_1(T')$.

We next claim that if  $P_0'$ is a component of $P'$  then $\pi_1(P_0')$ is a maximal abelian subgroup of $\pi_1(M')$.
If $P_0'$ is a component of $P$, then it is a maximal abelian subgroup of  $\pi_1(M)$, 
see \cite[Lemma 5.1.1]{canary-mccullough}, hence a maximal abelian subgroup of $\pi_1(M')\subset\pi_1(M)$.
If $\pi_1(P_0')$ is not a maximal abelian subgroup of $\pi_1(M)$, then it 
is a component of $\partial \mathcal N(A_0)$ where $A_0$ is either a component of 
the frontier of a solid or thickened torus or the boundary of a regular neighborhood of an essential M\"obius band,
see Johannson \cite[Lemma 32.1]{johannson}.
In all these cases, $P_0'$ is a component of the frontier of a 
component $X_0$ of $M-M_\mathcal A$ so that 
$\pi_1(P_0')$ is a proper subgroup of the abelian group $\pi_1(X_0)$.  If $\pi_1(P_0')$ is not a maximal abelian
subgroup of $\pi_1(M')$, then there exists an  element  $\alpha$ of
$\pi_1(M')\ssm \pi_1(P_0')$ which centralizes $\pi_1(P_0')$.
If $\beta$ is an element of $\pi_1(X_0)\ssm\pi_1(P_0')$,
then, by the Seifert-van Kampen Theorem, $\alpha$ and $\beta$ both centralize
$\pi_1(P_0')$ but do not themselves commute. However, this is
impossible in a  torsion-free discrete subgroup of ${\rm PSL}(2,\mathbb C)$.

Now consider a $\pi_1$-injective map $f:(A,\partial A)\to (M',P')$  of an annulus $A$
into $(M',P')$.
Notice that, by construction, no two distinct components of $P'$
contain homotopic curves, so $f(\partial A)$ is contained in a single component  $B$ of $P'$.
If $B$ is a torus then, since $(M,P)$ is pared,  
$f$ can be admissibly homotoped in $(M,P)$ to a map
with image in $B$. Since the frontier of $M'$ consists of essential annuli, one may alter the homotopy
to lie entirely within $(M',P')$. (One may see this by lifting the homotopy to the cover of $M$ associated
to $\pi_1(M')$ and noticing that the cover deformation-retracts onto $M'$.)
If $B$ is an annulus,
note that both components of $\partial A$ map to homotopic curves in
$B$: if not then the core of $B$ would have two distinct but conjugate
powers, but this is impossible in a torsion-free discrete subgroup of $PSL(2,\C)$
since the stabilizer of the fixed points of an element must commute
with it.
A homotopy between these two curves, adjoined to $f$, gives a map $\bar f$ of a
torus into $M'$. Since, $\pi_1(B)$ is a maximal abelian subgroup of $\pi_1(M')$,
the map $\bar f$, and hence $f$ itself, lifts to the cover of $M'$ associated to $\pi_1(B)$.
Since this cover deformation retracts to the lift of $B$, the lift of $f$ admissibly homotopes into  the lift of $B$,
which implies that $f$ itself is admissibly homotopic into $B$.
This completes the proof that $(M',P')$ is a pared manifold.

Since every component of $\partial_0M'-P'$ is an incompressible subsurface of a component of $\partial_0M$
and $\partial_0M$ is incompressible in $M$, $(M',P')$ has pared incompressible boundary.
\end{proof}

\subsection{Partial Convergence and Geometric Limits}
\label{partial convergence}

We now recall standard notation on algebraic and geometric convergence,
as well as introducing terminology for 
sequences of representations which
may not converge, but converge (in various senses) on 
subgroups. We take care, in Lemmas \ref{converge and basepoint} and
\ref{basepoints and limits}, to
keep track of basepoints, keeping in mind
the situation where there may be more than
one geometric limit of interest.

\subsubsection*{Partial convergence and basepoints}
Fix a group $H$. If $\rho\in D(H)$  and $g\in \PSL
2(\mathbb C)$, we let $\rho^g$ be the conjugate representation 
$h \mapsto g\rho(h) g^{-1}$. 
For $h\in H$ we let \hbox{$I_h:H\to H$} be conjugation by $h$ and note
that $\rho\circ I_h = \rho^{\rho(h)}$.

We call a group {\em nonelementary} if it is not virtually abelian.
For a sequence $\{\rho_n\}$ in $D(H)$ and a nonelementary subgroup $J<H$, we need two
distinct notions of partial convergence of the sequence on $J$. 

\begin{itemize}
\item We say $\{\rho_n\}$ {\em converges up to conjugacy on $J$} if there exists a
sequence $\{g_n\}$ in $\PSL 2 (\mathbb C)$ such that $\{\rho_n^{g_n}|_J\}$ converges in $D(J)$. 
\item
We say $\{\rho_n\}$ {\em converges up to inner automorphism on $J$} if there
exists a sequence $\{h_n\}$ in $H$ such that $\{\rho_n^{\rho_n(h_n)}|_J\}$
(equivalently $\{\rho_n\circ I_{h_n}|_J\}$)
converges in $D(J)$. 
\end{itemize}

Note that both of these definitions depend only on $\rho_n$ and the conjugacy class of
$J$ in $H$. The first definition also depends only on the conjugacy class of $\rho_n$,
i.e. on $[\rho_n]\in AH(H)$.  

Note that, if both  $\{\rho_n|_J\}$ and
$\{\rho_n^{g_n}|_J\}$ converge, then $g_n$ must remain in a compact subset
of $\PSL 2(\C)$. 
This is a consequence of the assumption that $J$ is
nonelementary, and the fact that the action of $\PSL 2(\C)$ by
conjugation on its nonelementary discrete subgroups is proper.
It follows from this that, if $\{\rho_n\}$ converges up to conjugacy on
$J$, then the limit representation is determined
up to conjugacy.

A representation $\rho\in D(H)$ determines a
basepoint $b_\rho$ in $N_\rho = \Hyp^3/\rho(H)$, namely the image of
the orbit $\rho(H)\zero$, where $\zero$ is a fixed basepoint for
$\Hyp^3$. We note that this basepoint is invariant under inner
automorphisms since $\rho$ and $\rho\circ I_h$ have the same image
for any $h\in H$. We can now relate behavior of basepoints to these
notions of partial convergence. 

\begin{lemma}{converge and basepoint}
Suppose that a sequence $\{\rho_n\}$ in $D(H)$ converges up to inner automorphism on a
nonelementary subgroup $J$ of $H$. For a sequence $\{g_n\}$ in $\PSL 2(\C)$,
\begin{enumerate}
\item if the sequence   $\{d(b_{\rho_n},b_{\rho_n^{g_n}})\}$ is bounded,
  then a subsequence of $\{\rho_n^{g_n}\}$ also converges up to inner automorphism on $J$.
\item if  $\{\rho_n^{g_n}\}$ converges up to inner automorphism on
    $J$, then   $\{d(b_{\rho_n},b_{\rho_n^{g_n}})\}$ is bounded.
\end{enumerate}
\end{lemma}

\begin{proof}
To compare the basepoints of two conjugate
representations $\rho$ and $\rho^g$, consider the cover $\Hyp^3 \to
N_{\rho}=N_{\rho^g}$ induced by the action 
of $\rho(H)$. The lift of $b_\rho$ to this cover is the orbit
$\rho(H)\zero$, whereas the lift of $b_{\rho^g}$ to this cover 
is the orbit $\rho(H) g^{-1}\zero$ (since it is the preimage by the
conjugating map $g$ of the orbit $\rho^g(H)\zero$).
Thus, $d_{N_\rho}(b_\rho,b_{\rho^g})$ is equal to the distance from
$\zero$ to $\rho(H)g^{-1}\zero$.

Now, after suitable conjugation we may assume that $\{\rho_n\}$ already
converges on $J$. 
If the distances
$d(b_{\rho_n},b_{\rho_n^{g_n}})$ are bounded, then 
$d_{\Hyp^3}(\zero,\rho_n(H)g_n^{-1}\zero)$ is bounded, and so there is a
sequence $\{h_n\}$ in $H$ such that $\{\rho_n(h_n)g_n^{-1}\}$ remains in a compact
subset of $\PSL 2(\C)$. Hence a subsequence of $\{\rho_n^{g_n\rho_n(h_n^{-1})}|_J\}$
converges.

For the second part, if  $\{\rho_n^{g_n\rho_n(k_n)}|_J\}$ converges
for some sequence $k_n\in H$, then $\{g_n\rho_n(k_n)\}$ lies in a
compact set as we argued above, and hence the distances
$\{d(b_{\rho_n},b_{\rho_n^{g_n}})\}$ are bounded. 
\end{proof}

These observations allow us to analyze the situation where
one sequence $\{\rho_n\}$ converges up to inner automorphism on several
subgroups, in terms of distances between their basepoints. 
Given  $\{\rho_n\in D(H)\}$ and a collection $\JJ$ of nonelementary
subgroups of $H$, if $\{\rho_n\}$ converges up to conjugacy
on each $J\in\JJ$, fix a conjugating sequence $\{g_n^J\}$
for which $\{\rho_n^{g_n^J}|_J\}$ converges, 
and let $b_n^J$ denote the sequence of associated basepoints.
Lemma \ref{converge and basepoint}
implies that the property that $d(b_n^J,b_n^L)$ is bounded for
$J,L\in\JJ$ is independent of the choice of conjugating sequences. We
now record the following consequence:

\begin{lemma}{basepoints and limits}
Let $\JJ$ be a collection of nonelementary subgroups of
$H$ and suppose that  a sequence 
$\{\rho_n\}$ in $D(H)$ converges up to conjugacy on each $J\in\JJ$.
With choice of basepoints as above, if $\{d(b^J_n,b^L_n)\}$ is bounded for
each $J,L\in\JJ$, then there is a single conjugating sequence $\{g_n\}$
such that, after possibly passing to a subsequence,  
$\{\rho_n^{g_n}\}$ converges up to inner automorphism on each
$J\in\JJ$.

If $b_n$ is the basepoint associated to $\rho_n^{g_n}$, then the sequence
$\{d(b_n,b^J_n)\}$ is bounded for each $J\in\JJ$.
\end{lemma}

\begin{proof}
Fixing one member $J_0 \in \JJ$ let $g_n=g_n^{J_0}$ so that
$\{\rho_n^{g_n}|_{J_0}\}$  converges. 
For any other $J\in\JJ$,
we have a bound on $d(b_n^{J_0},b_n^J)$,  so by Lemma \ref{converge
  and basepoint} we find that (a subsequence of) $\{\rho_n^{g_n}\}$ also converges up to
inner automorphism on $J$. The final statement follows since $b_n =
b_n^{J_0}$.
\end{proof}

\subsubsection*{Geometric convergence}
We say that a sequence $\{ \Gamma_n\}$ of torsion-free discrete subgroups of  $\PSL 2(\C)$
{\em converges geometrically} to a discrete group $\Gamma$, if it converges as a
sequence of closed subsets in the sense of Gromov-Hausdorff convergence.
It follows that the sequence $\{\mathbb H^3/\Gamma_n\}$ of quotient manifolds
converges geometrically to $N=\mathbb H^3/\Gamma$
(see Canary-Epstein-Green \cite{CEG} for an extensive discussion).
The following lemma recalls standard properties  of geometric convergence which will be used throughout the paper.

\begin{lemma}{basic geometric limit facts}
Suppose that a sequence  $\{\rho_n\}$  in $D(H)$ converges on a nonelementary
subgroup $J$ of $H$ to some $\rho\in D(J)$. Then after possibly restricting to a subsequence, $\{\rho_n(H)\}$
converges geometrically to a discrete nonelementary group $\hhat\Gamma$.
Let $\hhat N = \Hyp^3/\hhat \Gamma$ and let $\pi:N_\rho \to \hhat N$ 
be the natural covering map. Given $\ep\le \ep_0$ there exists a
nested sequence $\{Z_n\}$ of compact submanifolds exhausting $\hhat  N$ and 
$K_n$-bilipschitz smooth embeddings  $\psi_n:Z_n\to N_{\rho_n}$ such that:
\begin{enumerate}
\item 
$K_n\to 1$.
\item
$\psi_n\circ\pi$ carries the basepoint $b_{\rho}$ of $N_\rho$ to the
basepoint $b_{\rho_n}$ of $N_{\rho_n}$.
\item 
If $Q$ is a compact subset of a component of $\partial\hhat N_{(0,\ep)}$, then, for all
large enough $n$, $\psi_n(Q)$ is contained in $\partial(N_{\rho_n})_{(0,\ep)}$
and \hbox{$\psi_n(Z_n\cap (\hhat N- \hhat N_{(0,\ep)}))$} does not intersect $(N_{\rho_n})_{(0,\ep)}$ .
\item
If $X$ is a finite complex and $h:X\to N_\rho$ is continuous, then, for all
large enough $n$, $(\psi_n\circ \pi\circ h)_*$ is conjugate to $\rho_n\circ \rho^{-1}\circ h_*$.
    
\end{enumerate}
\end{lemma}

\noindent
{\bf Remark:} The existence of the exhausting sequence of submanifolds with properties (1) and (2) is 
discussed in \cite[Cor. I.3.2.11]{CEG}. Property (3) is established in \cite[Lemma 2.8]{ELC2}. Property (4)
is an immediate consequence of the proof of \cite[Prop. 3.3]{canary-minsky} which it generalizes.

\subsection{Hierarchies, model manifolds, and topological order}
\label{hierarchies etc}

We briefly review here some notation and definitions that arise in
\cite{masur-minsky2,ELC1,ELC2}. See also \cite{BBCM} for a
summary. This machinery will play a central role in Section \ref{length bounds}. 

\subsubsection*{Markings, subsurface projections and hierarchies}
A {\em marking} $\mu$ on a surface $S$, as in \cite[Section 2.5]{masur-minsky2}, is a multicurve, denoted
$\base(\mu)$,  together with at most one  transversal curve 
for each component $\beta$ of $\base(\mu)$ (i.e. a curve intersecting $\beta$ minimally which is disjoint from
$\mu-\beta$). A {\em generalized marking} $\mu$, as in \cite[Section 5.1]{ELC1},  is a geodesic
lamination $\base(\mu)$ which supports a measure, together with at most one transversal curve for
each closed curve component of $\base(\mu)$.

To an end invariant $\nu$ on a surface $S$, as
described in Section \ref{deformation}, we associate in \cite[Section 7.1]{ELC1} a 
generalized marking $\mu$ as follows: $\base(\mu)$ consists of the
parabolic locus of $\nu$, the ending laminations of geometrically infinite components
of the complement of the parabolic locus and a minimal length pants decomposition of the hyperbolic
structure on each geometrically finite component of $\nu$. To each curve in $\base(\mu)$ which is
non-peripheral in a geometrically finite component we choose a minimal length transversal curve.

Recall the {\em curve complex} $\CC(Y)$ of a surface, whose
vertices are homotopy classes of essential simple curves (except when
$Y$ is an annulus and a special definition is needed), see \cite{masur-minsky}. 
If $Y$ is an essential subsurface of $S$ and $\alpha\in\CC(S)$ essentially intersects $Y$, one defines
$\pi_Y(\alpha)\in \CC(Y)$ by taking a component of $\alpha\cap Y$ and combining with arcs in $\partial Y$ 
to obtain a (coarsely well-defined) element of $\CC(Y)$ (again a special definition is needed when $Y$ is an
annulus). The curve complex and {\em subsurface projections} $\pi_Y:\CC(S)\to\CC(Y)$ are studied
in \cite{masur-minsky,masur-minsky2} and elsewhere. We also note that
$\pi_Y(\mu)$ is well defined when $\mu$ is a generalized marking, and
by extension so is $\pi_Y(\nu)$ for an end invariant $\nu$. 

A {\em hierarchy}, as in \cite{masur-minsky2} and \cite{ELC1}, is a collection of 
{\em  tight geodesics} in curve complexes of subsurfaces of $S$, with
certain interlocking properties.  A tight geodesic in $\CC(W)$ is a
sequence of simplices $(w_n)$ such that whenever $v_j\in w_j$ are
vertices we have $d(v_i,v_j) = |i-j|$ for $i\ne j$, and such that
$w_i$ is the boundary of the subsurface filled by $w_{i-1}\union
w_{i+1}$. 

To a pair $(\mu^+,\mu^-)$ of generalized markings (which share no infinite leaves), one can associate a
hierarchy $H(\mu^+,\mu^-)$, see \cite[Section 5.5]{ELC1}.
Given a subsurface $W\subseteq S$, we define
 $$d_W(\mu^+,\mu^-) \equiv  d_{\CC(W)}(\pi_W(\mu^+),\pi_W(\mu^-)).$$
If  $d_W(\mu^+,\mu^-)$ is sufficiently large (where the threshold depends only on $S$)
then $\CC(W)$ supports a (unique) geodesic $h_W\in H$, whose initial vertex
 $i_W$ (respectively terminal vertex $t_W$) is uniformly close in
$\CC(W)$ to $\pi_W(\mu^-)$ (resp.  $\pi_W(\mu^+)$), see \cite[Lemma 6.2]{masur-minsky2} and
\cite[Lemma 5.9]{ELC1}.
Let $\CH\subset\CC(S)$ denote the set of all curves appearing in geodesics in $H$.

\subsubsection*{Model manifolds}
In \cite[Section 8]{ELC1} we associate to the hierarchy $H = H(\mu^+,\mu^-)$ a
{\em model manifold} $M(\mu^+,\mu^-)$ which is a copy of ${\rm int}(S)\times\R$
endowed with a certain metric and a subdivision into blocks and tubes.
We will mostly be concerned with the tubes. 
There is a tube $U(\gamma)$ for each curve $\gamma$ in $\CH$ and each component of $\partial S$.
Each tube $U(\gamma)$ has the form
$A\times J$ where $A$ is an open regular neighborhood of $\gamma$ in $S$ and $J$
is an interval, where $J=\mathbb R$ if $\gamma$ is a component of $\partial S$.
In particular, the complement $\hhat M(\mu^+,\mu^-)$ in $M(\mu^+,\mu^-)$ of the tubes associated to $\partial S$
is homeomorphic to $S\times\mathbb R$.

The following result is  the
main outcome of \cite{ELC1} and \cite{ELC2}.

\medskip\noindent
{\bf Bilipschitz Model Manifold Theorem:} (\cite{ELC1,ELC2}) {\em If $S$ is a compact surface, there exists
$L_h(S)>1$ and $\epsilon_h(S)>0$ such that if $\rho\in D(S)$ has end invariants $(\nu^+,\nu^-)$ and
associated generalized markings $(\mu^+,\mu^-)$, then there exists a $L_h(S)$-bilipschitz map
$$f_\rho:N_\rho\to M(\mu^+,\mu^-)$$ 
such that if $\alpha\in\CC(S)$ and $\ell_\rho(\alpha)<\epsilon_h(S)$, then $\alpha\in \CH$ and
$f_\rho(U(\alpha))=\MT(\alpha)$.}

\medskip

By construction the core curves
of the tubes of $M$ have length at most 1, so 
$L_h$ also bounds the
length of their images in $N_\rho$, i.e. if $\alpha\in\CH$, then
$\ell_\rho(\alpha)\le L_h(S)$.

\subsubsection*{Ordering}
If $A$ and $B$ are curves and/or subsurfaces of $S$ and $f:A\to S\times \mathbb R$ and $g: B\to S\times\mathbb R$ 
are maps which are homotopic to the inclusions of $A$ and $B$ into $S\times\{0\}$, then 
we say that $f$ {\em lies above} $g$ if $f$ is homotopic to $+\infty$ in the complement of $g(B)$
(i.e. there is a proper map $F:A\times [0,\infty)\to S\times [r,\infty)$, for some $r$, such that $f=F(\cdot,0)$, whose image is
disjoint from $g(B)$). We similarly 
say that  $f$ {\em lies below}  $g$ if $f$ is homotopic to $-\infty$ in the complement of $g(B)$. 
(This topological ordering is discussed extensively in \cite[Section 3]{ELC2}.)

We will often
identify a curve $\alpha$ in $S\times \mathbb R$ with a function whose image is $\alpha$.
Notice that a curve $\alpha$ may lie above a curve $\beta$ without $\beta$ lying below $\alpha$.
We say that a curve $\alpha$ is {\em unknotted} if it is isotopic to $\hhat\alpha\times \{0\}$ for some
curve $\hhat\alpha$ on $S$.

A {\em level surface} in $S\times\mathbb R$ is an embedding $f:(S,\partial S) \to (S\times \mathbb R,\partial S\times\mathbb R)$
which is properly isotopic to  the identification of $S$ with $S\times \{0\}$. 
Again, we often identify a level surface with its image
in $S\times\mathbb R$.

The following result records a special property of homotopy equivalences which will be used in Section \ref{length bounds}.

\begin{lemma}{order surface}{}
{\rm (\cite[Lemma 2.7]{BBCM})}
Let $\alpha$ and $\beta$ be simple closed curves on $S$ that intersect essentially.
Let $f: (S,\partial S) \to (S\times\mathbb R,\partial S\times\mathbb R)$ be a homotopy equivalence with image 
disjoint from $\beta \times \{0\}$. Then $f$ lies above (below) $\beta \times \{0\}$ if and only if $f|_\alpha$ lies above (below) $\beta \times \{0\}$.
\end{lemma}

We may apply this ordering to tubes in  $\hhat M(\mu^+,\mu^-)$.
If $v,w\in\CH$, we say that $U(v)$ lies above (below) $U(w)$ if a core curve of $U(v)$
lies above (below) a core curve of $U(w)$ in $\hhat M(\mu^+,\mu^-)$. By construction, if $v$ and $w$ both lie
in simplices of a geodesic $h_W$ in $H(\mu^+,\mu^-)$, then $U(v)$ lies above (below) $U(w)$ if and only if
the simplex containing $v$ is not adjacent to and occurs after (before) the simplex containing $w$ on $h_W$. 
(If $W$ is a once-punctured torus or thrice-punctured sphere, then tubes associated to adjacent vertices are
also ordered consistently.)

\subsection{Thurston's Relative Compactness Theorem and its consequences}
\label{relative}

The goal of this section is Corollary \ref{ourrelcompactness}, which 
shows that any sequence in $AH(M,P)$ may be remarked by
homeomorphisms supported on ${\rm window}(M,P)$, so that it converges,
up to subsequence, on $M_{\mathcal A}$ for some  robust collection $\mathcal A$ of essential 
annuli in $(M,P)$. The core curves of the components of $\AAA$ have
length 0 in the limit representations, and
$\AAA$ can be chosen to be maximal in the sense that
the core curve of any essential annulus in a component of
$(M_\mathcal A,P_\mathcal A)$ has positive length in the
associated limit representation.

Our major tool is 
Thurston's Relative Compactness Theorem \cite{thurstonIII}\footnote{Thurston's Relative Compactness Theorem
is a generalization of his earlier result that $AH(M)$ is compact if $M$ is acylindrical, see Thurston \cite{thurston1}.
The proof combines a uniform bound on lengths of curves in the boundary  of the window base, see Thurston 
\cite[Thm. 1.3]{thurstonIII} or Morgan \cite[Thm. A.1]{morgan-dimn}, and a compactness theorem for representations
of acylindrical pared 3-manifold groups, see Thurston \cite[Thm. 3.1]{thurstonIII}
or Morgan-Shalen \cite[Thm. 2.2]{morgan-shalenIII}.}
(see also Morgan \cite[Cor. 11.5]{Morgan}):

\begin{theorem}{relcompactness} {\rm (Thurston \cite{thurstonIII})}
Suppose that $(M,P)$ is a pared manifold with pared incompressible boundary and
$R$ is  a component of $M\ssm{\rm window}(M,P)$. The restriction mapping from 
$AH(M,P)$ to $AH(R)$ has bounded image.
\end{theorem}

If $X$ is a submanifold of $M$ whose frontier is a collection of essential annuli in $(M,P)$,
then each component of $X$ determines a conjugacy class of subgroups of $\pi_1(M)$.
We say that a sequence $\{\rho_n\}\subset D(M,P)$ {\em converges up to conjugacy on $X$}
if $\{\rho_n\}$ converges up to conjugacy on $\pi_1(R)$ (in the sense
of Section \ref{partial convergence}) for each component $R$ of $X$.
In this language, Theorem \ref{relcompactness} implies that any sequence in $D(M,P)$
has  a subsequence which converges up to conjugacy on \hbox{$M\ssm{\rm window}(M,P)$}.

Canary, Minsky and Taylor \cite[Thm. 5.5]{CMT} observed that, if one allows oneself to remark 
by pared homeomorphisms supported on $\Sigma(M,P)$, then one can find a subsequence which
converges up to conjugacy on the complement of a robust collection of essential annuli.

We will adopt the following convenient notational convention throughout the paper.
If \hbox{$\rho\in AH(M,P)$} and $\mathcal B$ is a collection of essential annuli in $(M,P)$, then
$\ell_\rho(\mathcal B)$ denotes the sum of the lengths of the geodesic representatives in $N_\rho$
of the core curves of annuli in $\mathcal B$.

\begin{theorem}{CMTrelcompactness}{\rm (Canary-Minsky-Taylor \cite{CMT})}
Let $(M,P)$ be a pared 3-manifold with pared incompressible boundary. If $\{\rho_n\}$ is a sequence
in $D(M,P)$, then, after passing to a subsequence,
there is a robust collection $\mathcal B$ of
essential annuli in $(M,P)$ and
a sequence of pared homeomorphisms \hbox{$\{\phi_n:(M,P)\to (M,P)\}$}, each supported on ${\rm window}(M,P)$, such that
\begin{enumerate}
\item
$\lim\ell_{\rho_n\circ(\phi_n)_*}(\mathcal B)=0$, and 
\item 
$\{\rho_n\circ (\phi_n)_*\}$ converges up to conjugacy on $M_\mathcal B$.
\end{enumerate}
\end{theorem}

We provide a sketch of the proof of Theorem \ref{CMTrelcompactness}, since our
statement is slightly different and more general than the one given in \cite{CMT}, although their
proof goes through directly to yield our statement.

\medskip\noindent
{\em Sketch of proof of Theorem \ref{CMTrelcompactness}:} 
We first apply Theorem \ref{relcompactness} to pass  to a subsequence so that
$\{\rho_n\}$ converges up to conjugacy on $M\ssm{\rm window}(M,P)$.
We  then construct $\mathcal B$ and $\{\phi_n\}$
piece by piece. We first  include the frontier of any thickened torus component of $\Sigma(M,P)$ in $\mathcal B$.
If $V$ is a solid torus component of $\Sigma(M,P)$ with core curve $v$,
then we include ${\rm Fr}(V)$ in $\mathcal B$ if and only if $\lim \ell_{\rho_n}(v)=0$.

Let $F$ be the collection of components of ${\rm wb}(M,P)$ which are base surfaces of interval
bundle components of $\Sigma(M,P)$.
For each $n$, let $F_n$ be obtained by removing any boundary component $f$ so that $\ell_{\rho_n}(f)=0$.
Consider a pleated surface \hbox{$f_n:(F_n,\tau_n)\to N_{\rho_n}$} in the homotopy class of
the inclusion map, where $\tau_n$ is
a finite area hyperbolic metric on $F_n$.
(A pleated surface is a $1$-Lipschitz map which is totally geodesic on a geodesic lamination which includes
the boundary, called the pleating locus,
and totally geodesic on the complement of the pleating locus.)
One may pass to a subsequence so that there exists a two-sided multicurve 
$b_I$ on $F$ and a sequence $\{\hat\psi_n\}$ of
homeomorphisms of ${\rm int}(F)$, which extend to the identity on $\partial F$, 
so that  $\lim\ell_{\hat\psi_n^*(\tau_n)}(b_I)=0$ and if $y$
is a curve on $F$ which is disjoint from $b_I$, then $\{\ell_{\hat\psi_n^*(\tau_n)}(y)\}$ is bounded.
(See \cite[Prop. 5.6]{CMT} and its use in the proof of \cite[Thm. 5.5]{CMT} for more details.)
We complete the construction of $\mathcal B$ by adding the interval bundle over $b_I$ to $\mathcal B$.

We may extend each $\hat\psi_n$ to a homeomorphism of ${\rm window}(M,P)$ which is the
identity on ${\rm Fr}({\rm window}(M,P))$ and hence to a pared homeomorphism $\psi_n$ of $(M,P)$ which
is supported on ${\rm window}(M,P)$. One then checks that, up to subsequence, there
exists a sequence $D_n$ of Dehn multitwists in ${\rm Fr}({\rm window}(M,P))$ so that  if $\phi_n=D_n\circ\psi_n^{-1}$,
then $\{\rho_n\circ(\phi_n)_*\}$ converges up to conjugacy on $M_\mathcal B$, which verifies property (2)
(see \cite[Lem. 5.7]{CMT} and its use in the proof of \cite[Thm. 5.5]{CMT} for more details).
Properties (1) and (3)  hold by construction.
\qed

\medskip

It will be useful to be able to choose the robust collection of annuli $\mathcal B$ in 
Theorem \ref{CMTrelcompactness} to be maximal in
the sense that the length of any essential annulus in $(M_{\mathcal B},P_{\mathcal B})$ has positive
length in its associated limit representation.

\begin{corollary}{ourrelcompactness}
Let $(M,P)$ be a pared 3-manifold with pared incompressible boundary.
If  $\{\rho_n\}$ is a sequence in $D(M,P)$, then, after passing to a subsequence,
there is a robust collection $\mathcal A$ of
essential annuli  in $(M,P)$ and
a sequence of pared homeomorphisms $\{\phi_n:(M,P)\to (M,P)\}$, each supported on ${\rm window}(M,P)$, such that
\begin{enumerate}
\item
$\lim\ell_{\rho_n\circ(\phi_n)_*}(\mathcal A)=0$
\item 
$\{\rho_n\circ(\phi_n)_*\}$ converges  up to conjugacy on $M_\mathcal A$, and
\item 
If $B$ is an essential annulus in a component of $(M_\mathcal A,P_\mathcal A)$,
then

$\lim \ell_{\rho_n\circ(\phi_n)_*}(B)>0$.
\end{enumerate}
\end{corollary}

\begin{proof}{} 
Theorem \ref{CMTrelcompactness} guarantees that there exists a subsequence,
still called $\{\rho_n\}$, a robust collection $\mathcal B$  of
essential annuli  with base multicurve $b$ and sequence of pared homeomorphisms $\{\phi_n:(M,P)\to (M,P)\}$
such that
$\lim\ell_{\rho_n\circ(\phi_n)_*}(b)=0$ and
$\{\rho_n\circ (\phi_n)_*\}$ converges up to conjugacy on $M_\mathcal B$.
We may further assume that if $V$ is a solid torus component of $\Sigma(M,P)$ with core curve $v$, 
then ${\rm Fr}(V)\subset \mathcal B$ if and only if $\lim \ell_{\rho_n}(v)=0$

Let $\mathcal C$ be a maximal collection of disjoint, non-parallel essential annuli in $(M_\mathcal B,P_\mathcal B)$
such that \hbox{$\lim \ell_{\rho_n\circ(\phi_n)_*}(\mathcal C)=0$}.  
(Alternatively, we may choose $c$ to be a maximal two-sided multicurve in ${\rm wb}(M,P)-b$
such that $\lim\ell_{\rho_n\circ(\phi_n)_*}(c)=0$ and let $\mathcal C$ be the interval bundle over $c$.)
Let $\mathcal A$ be the union of $\mathcal B$ and $\mathcal C$.
Properties (1) and (2)  hold for $\mathcal A$, since they held for the subcollection $\mathcal B$ of $\mathcal A$.
Property (3) holds by construction.
\end{proof}

\section{Length bounds}
\label{length bounds}

Let $S$ be a compact surface and $[\rho]\in AH(S)$ a Kleinian surface group. 
The goal of this section is a 
criterion for bounding the length at infinity of a curve in $N_\rho$,
in terms of the situation of a representative
of the curve in the interior. That is, given a simple curve $\alpha\subset S$
with a representative $\hhat \alpha$ in $N_\rho$ of given length, we wish to know
when we can bound $\ell_{\nu^-(\rho)}(\alpha)$. The key is to study
the collection of bounded length geodesics in $N_\rho$ that lie
``between'' $\hhat\alpha$ and the bottom end of $N_\rho$, in the
following precise sense.

Let  $\mathcal C(\hhat\alpha,L)$ denote the set of curves on $S$ which intersect
$\alpha$ essentially and whose geodesic representatives in
$N_\rho$  have length at most $L$ and do not lie above $\hhat\alpha$
(in the sense of Section \ref{hierarchies etc}).
Parabolic curves have no geodesic representatives, but we adopt the
convention that  $\beta$ lies above every curve in $N_\rho$ if it is in the
parabolic locus of $\nu^+$, and below if it is in the parabolic locus
of $\nu^-$.

We will see that given an upper bound on the length of $\hhat\alpha$,
an upper bound on the size of $\mathcal C(\hhat\alpha,L)$ and a lower bound on
the length of any curve in $\mathcal C(\hhat\alpha,L)$, one obtains an upper bound on the length of $\alpha$
in the bottom ending invariant. In the statement, $L_h(S)$ is the constant from 
the Bilipschitz Model Manifold Theorem in Section \ref{hierarchies etc}.

\begin{theorem}{upper bound}{}{}
Let $S$ be a compact surface.  Given $R$, $L_0$ ,$\epsilon\in(0,\mu_3)$ and \hbox{$L\ge L_h(S)$}, 
there exists $L_1>0$ such that, if $\rho\in AH(S)$,
$\alpha$ is a simple closed curve on $S$, and $\hhat \alpha$ is a representative of
$\alpha$ in $N_\rho$ such that

\begin{enumerate}
\item $\hhat\alpha$ has length at most $L_0$,
\item $\mathcal C(\hhat\alpha,L)$  contains at most $R$ elements, 
\item $\mathcal C(\hhat\alpha,L)$ contains no curves of length less than $\ep$, and
\item $\hhat\alpha$ lies on a level surface $F$ which does not intersect  $(N_\rho)_{(0,\ep)}$,
\end{enumerate}
then 
$$l_{\nu^{-}(\rho)}(\alpha) < L_1.$$
\end{theorem}

The idea of the proof is the following: If $\alpha$ has large, or infinite,
length in $\nu^-(\rho)$, then there
must be a subsurface $W$ in $S$ so that the projection distance
$d_W(\alpha,\nu^-(\rho))$ is large or infinite. This forces the hierarchy $H(\mu^+(\rho),\mu^-(\rho))$
to have a long geodesic $h_W$ associated to $W$, which corresponds to a large
region in the model manifold of $N_\rho$ containing a large number of
bounded-length curves (corresponding to hierarchy curves) isotopic
into $W$. In the model manifold, the topological placement of these
curves corresponds to their location along the hierarchy geodesic.
Thus, using what we know about the projection of $\alpha$ to $h_W$,
we conclude that a large number of the geodesic representatives of these curves must not
lie above  $\hhat\alpha$ in $N_\rho$ -- hence  must belong
to $\mathcal C(\alpha,L)$. The hypotheses bound the number of
such curves and therefore the size of $d_W(\alpha,\nu^-(\rho))$. The
main technical difficulties in the proof involve converting the
relatively nice picture in the model manifold to the slightly messier
arrangement of true geodesic representatives in $N_\rho$.

\begin{proof}
We  first observe that $\alpha$ lies in the thick part of a Riemann surface component of $\nu^-(\rho)$.

\begin{lemma}{no ending lamination}{}{}
With the same assumptions as Theorem \ref{upper bound},
there exists $\ep' < \ep$ such that the curve $\alpha$ is homotopic into the $\epsilon'$-thick part
$Z'$ of a Riemann surface component $Z$ of $\nu^-(\rho)$.
\end{lemma}

\begin{proof}
If a component $p$ of the parabolic locus of $\nu^-(\rho)$ 
crosses $\alpha$ essentially, then the cusp associated to $p$ is downward-pointing
and hence $p$ lies below $\alpha^*$. However, this contradicts the fact that 
$\mathcal C(\alpha,L)$ contains no curves of length less than $\ep$.
If $\alpha$ essentially intersects a subsurface  $W$ which supports an ending
lamination in $\nu^-(\rho)$, then there exists a sequence $\{\beta_n\}$ of hierarchy curves
so that $\{\beta_n^*\}$ exits the downward-pointing simply degenerate end with base surface $W$.
However, this contradicts the fact that $\mathcal C(\hhat\alpha,L)$  contains only finitely many curves.
Therefore, $\nu^-(\rho)$ has a
Riemann-surface component supported on a subsurface $Z$ containing
$\alpha$. 

Choose $\delta=\delta(L_0)>0$ so that, if a homotopically non-trivial curve $\gamma$ in a hyperbolic 3-manifold
is of length at most $L_0$ and intersects 
$\MT_{\delta}(\beta)$ for some primitive curve $\beta$, then $\gamma$ is homotopic to a power of  $\beta$.

Suppose that $\beta$ is a curve such that $\ell_Z(\beta) < 2\ep'$ where $\ep' = \min(\delta(L_0),\ep)$.
Theorem 3.1 in Epstein-Marden-Markovic \cite{EMM} implies that $\beta$ has a representative
$\hhat \beta$ on the bottom boundary of the convex core which has length at most $2\ep'\le\ep$
So, $\hhat\beta$ lies inside $\MT_{\ep'}(\beta)$ and
one may then connect $\beta^*$ to $\hhat\beta$ by an annulus $A_0$ in the convex core which
is entirely contained within  $\MT_{\epsilon'}(\beta)$. We construct a homotopy $A$  from $\beta^*$ to $-\infty$,
by concatenating $A_0$ with an annulus in $r^{-1}(\hhat\beta)$ where $r$ is the nearest point retraction of $N_\rho$ 
onto its convex core. Since $r(\hhat\alpha)$ has length at most $L_0$, it must be disjoint from $\MT_{\ep'}(\beta)$
and hence from $A$. Therefore, $\beta^*$ lies below $r(\hhat\alpha)$. Since $r(A)=A_0$, $A$ is also disjoint
from $\hhat\alpha$, so $\beta^*$ also lies below $\hhat\alpha$. However, this contradicts our assumptions,
so $\alpha$ must be homotopic into the $\ep'$-thick part of $Z$.
\end{proof}

Let $\mu^-(\rho)$ be the generalized marking of $S$ which we associate with $\nu^-(\rho)$,
and let $\mu'$ be the restriction of $\mu^-(\rho)$ to $Z'$. 
Notice that, assuming $\ep'$ is small enough,  
$\partial Z'$ is contained in $\mu^-(\rho)$, so $\base(\mu')$ consists of the essential
curves in $Z'$ which are contained in $\base(\mu^-(\rho))$ and the transversals in $\mu'$ are
simply the transversals in $\mu^-(\rho)$ to the curves in $\base(\mu')$. Since $Z'$ is $\ep'$-thick,
both the base curves and the transversals in $\mu'$ have uniformly bounded length.
Therefore, $l_Z(\alpha)$
can be bounded in terms of the complexity of its intersection with
$\mu'$, and this in turn is controlled by the subsurface projections
of $\alpha$ and $\mu'$, as in Theorem 6.12 of \cite{masur-minsky2}
or Lemma 2.3 of \cite{BBCL}.
In particular in order to bound $l_Z(\alpha)$ it suffices to bound
\begin{equation}\label{W alpha bound}
  \sup_W \ d_W(\alpha,\mu') = \sup_W \  d_W(\alpha,\nu^-(\rho))
\end{equation}
where $W$ varies over all subsurfaces of $Z'$ intersecting $\alpha$. 

Theorem B of \cite{minsky:kgcc} tells us that when the diameter in $\CC(W)$
of the projections of curves in $S$ with a given length bound is
sufficiently large, $\boundary W$ must be short. In particular there
exists $A_0$,
depending only on $\ep'$ and on the length bounds on $\alpha$ and $\mu'$, such
that, if 
$$
d_W(\alpha,\mu') > A_0
$$
then  $\ell_\rho(\boundary W) <\ep'$. Since if $d_W(\alpha,\mu')\le A_0$ we
are done,  we may assume from now on that
$\ell_\rho(\boundary W) <\ep'\le \ep.$

Suppose now that $\alpha$ essentially intersects a component $\beta$ of
$\boundary W$ (in particular this must be the case if $W$ is an
annulus).
By assumption, since
$\ell_\rho(\beta)<\ep$, $\beta^*$ lies above $\hhat \alpha$. 
Since $\hhat\alpha$ lies on a level surface $F$ disjoint from $\beta^*$,
it follows that $\hhat\alpha$ lies below $\beta^*$.
Notice that one may homotope  $\hhat \alpha$ to $\alpha^*$ through
curves of length at most $L_0$.
None of the curves in
the homotopy can intersect $\MT_{\ep'}(\beta)$, since we can choose
$\ep' < \delta(L_0)$ (see the proof of Lemma \ref{no ending lamination}),
so $\alpha^*$ lies below $\beta^*$ as well.

Now Theorem 1.3 of \cite{BBCM} tells us that a geodesic of bounded length
that lies below $\beta^*$ (and overlaps it on $S$) must have
projection in $W$ close to the bottom end invariant. That is, there is a constant $A_1$
(depending on $L_0$ and $S$)
such that, if $\alpha^*$ lies below $\beta^*$ then
$$
d_W(\alpha,\nu^-) \le A_1. 
$$
Thus, we are done in this case.

The only remaining possibility is that $\alpha$ is a non-peripheral
curve contained in $W$. This case follows from the following lemma:

\begin{lemma}{thick surface}{}{}
Let $S$ be a compact surface.  Given $L_0$, $L\ge L_h(S)$, and $K$, there exists $D>0$ such that if
$\alpha$ is a curve on $S$ and 
$\hhat \alpha$ is a representative of
$\alpha$ in $N_\rho$ with length at most $L_0$,  $\alpha$ is  non-peripheral in an essential subsurface $W$ of $S$, and
$d_W(\alpha, \nu^-(\rho))>D$, then $\mathcal C(\hhat\alpha,L)$ contains at least $K$ curves.
\end{lemma}

\begin{proof}
Let $g:X\to N_\rho$ be a 1-Lipschitz surface that realizes $\alpha$ and $\partial W$, i.e. $X$ is a finite
area hyperbolic surface homeomorphic to $S$, $g$ is a 1-Lipschitz map in the homotopy class of $\rho$ and $g$ 
takes the geodesic representatives of $\alpha$ and $\partial W$ in $X$ to their geodesic representatives
in $N_\rho$ by an isometry. (For example, one may choose a pleated surface, see \cite[Chapter I.5.3]{CEG}).
By Lemma 3.1 in \cite{BBCM} there exists a multicurve $\Gamma$ of hierarchy curves, each of which has
length at most $L_2=L_2(S,L_0)$ on $X$
such that there are no hierarchy curves supported on the complement of $\Gamma$.

Theorem 1.2 in \cite{BBCM}  tells us that
the projection to $W$ of the set of curves of length at most $L_0$ in $N_\rho$ lies in
a uniformly bounded Hausdorff distance of a $\CC(W)$-geodesic joining
$\pi_W(\nu^+(\rho))$  to $\pi_W(\nu^-(\rho))$. So, there exists a constant $a=a(S,L_0)$
so that 
$$d_W(\nu^+(\rho),\nu^-(\rho))>d_W(\alpha,\nu^-(\rho))-a$$

Choose $\epsilon_2>0$ small enough that every curve on $S$ whose geodesic representative in $N_\rho$
has length at most $\epsilon_2$ lies in the hierarchy and such that any geodesic of length less than $\ep_2$ on
a hyperbolic surface cannot intersect another geodesic of length at most $L_2$.
If $d_W(\nu^+(\rho),\nu^-(\rho))$ is sufficiently large, then $W$ is a domain in the hierarchy, see \cite[Lemma 5.9]{ELC1}.
So, combining this fact with Theorem B from \cite{minsky:kgcc}, we see that
there exists  $D_1=D_1(S,\ep_2)$ so that if $d_W(\nu^+(\rho),\nu^-(\rho))>D_1$,
then $\ell_\rho(\partial W)<\ep_2$ and $W$ is a domain in the hierarchy.  We may assume that $D>D_1+a$, so that
this is the case.

Since $g$ is 1-Lipschitz and realizes
$\partial W$, no curve in $\Gamma$ can intersect $\partial W$ and $\partial W$ must be contained in $\Gamma$.
Moreover, since $W$ is a domain in the hierarchy,
there exists 
$\gamma\in\Gamma\cap \CC(W)$. Let $\gamma^X$ be a representative of $\gamma$ on $g(X)$ of
length at most $L_2$.

Let $M_\rho=M(\mu^+(\rho),\mu^-(\rho))$ and let $f_\rho:M_\rho\to N_\rho$ be the model map. 
If $\beta$ is a hierarchy curve, let $\beta^M$ be the image in $N_\rho$
of the core curve of $U(\beta)$. Since $\ell_X(\alpha^*)\le L_0$ and \hbox{$\ell_X(\gamma^X)\le L_2$},
\hbox{$d_W(\alpha,\gamma)<c=c(S,L_0,L_2)$}. Therefore, see \cite[Lemma 2.6]{BBCM}, there exists $b=b(S,L_0,L_2)$, so
that there are at least \hbox{$d_W(\alpha,\nu^-(\rho))-b$} curves $\beta$
on $h_W$  so that  $U(\beta)$ lies below 
$U(\gamma)$ in $M_\rho$. (In the simple situation when $\gamma$ lies on $h_W$, we may take $b=c+1$ and the curves
are simply curves in the vertices of $h_W$ which precede but are not adjacent to the vertex containing $\gamma$.)
\realfig{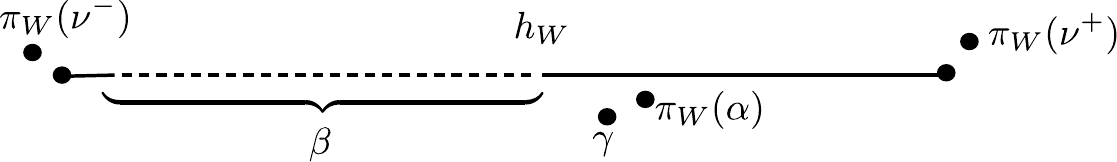}{When $\beta$ lies in the indicated range, on the
  part of $h_W$ to the left of $\gamma$, $U(\beta)$ lies below
  $U(\gamma)$ in the model $M_\rho$.}
If $U(\beta)$ lies below 
$U(\gamma)$ in $M_\rho$, then  $\beta^M$ lies below $\gamma^M$ in $N_\rho$.
In order to complete the proof, it suffices to bound from above the number of such curves whose
geodesic representatives lie above $\hhat\alpha$.

Let $Y$ be a ruled homotopy from $\alpha^*=\alpha^X$ to $\hhat\alpha$, and let
$Y'$ be a ruled homotopy from $\gamma^X$ to $\gamma^M$ (see Figure
\ref{beta-alpha}).
\realfig{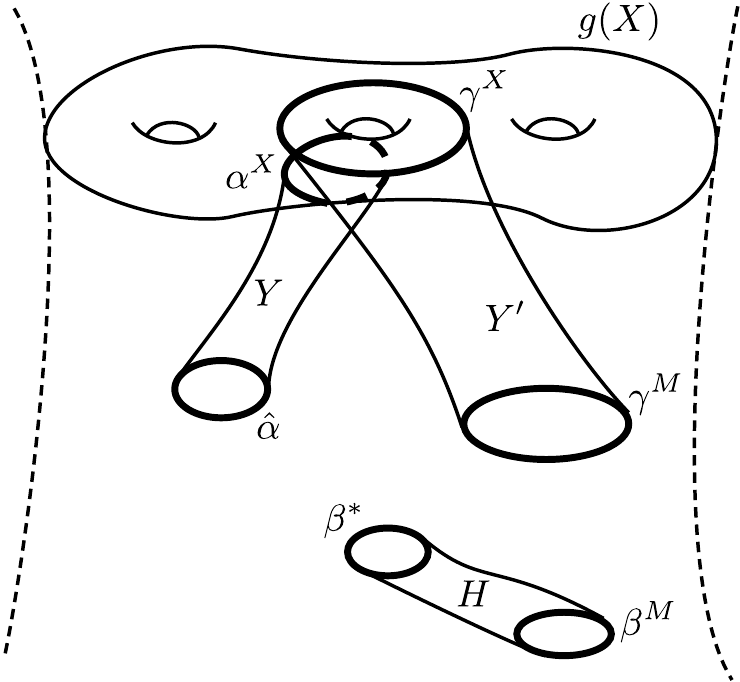}{If $H$ avoids $g(X) \union Y\union Y'$ then the
  ordering of $\hhat \alpha$ and $\beta^*$ is consistent with that of
  $\gamma^M$ and $\beta^M$.}
We claim that if $\beta$ is a curve in $h_W$ so that  $\beta$ overlaps $\alpha$, $\beta^M$ lies
below $\gamma^M$, and there is a
homotopy $H$ of $\beta^M$ to $\beta^*$ that is disjoint from 
$g(X)\union Y\union Y'$, then 
$\beta^*$ does not lie above $\hhat\alpha$.
We can prove this in a few steps.

\begin{enumerate}
\item Since $\beta^M$ is disjoint from $Y'$,  $\gamma^X$ lies above
  $\beta^M$. One concatenates $Y'$
  with the homotopy of $\gamma^M$ to $+\infty$, to obtain a homotopy
  of $\gamma^X$ to $+\infty$ disjoint from $\beta^M$.
\item Since $\beta^M$ is   disjoint from $g(X)$, $\alpha^X$ lies above
  $\beta^M$. Using Lemma \ref{order surface}, the fact that
  $\gamma^X$ lies above $\beta^M$ implies that $g$ lies above
  $\beta^M$, and therefore that $\alpha^X$ does as well. 
\item Since $\beta^M$ is disjoint from $Y$, $\hhat\alpha$ lies above
  $\beta^M$. As in (1), one concatenates $Y$
  with the homotopy of $\alpha^X$ to $+\infty$, to obtain a homotopy
  of $\alpha^X$ to $+\infty$ disjoint from $\beta^M$.
 
  \item Since $\hhat\alpha$ lies above $\beta^M$ and $\alpha$ overlaps $\beta$, $\beta^M$ does {\em
    not} lie above $\hhat\alpha$. Since $H$ is disjoint from
    $\hhat\alpha$, $\beta^*$ also does not lie above $\hhat\alpha$. 
\end{enumerate}

We next observe that
there is an upper bound $m(L_h)$ on the number of homotopy
classes of primitive curves of length at most $L_h$ that intersect $g(X)\cup Y\union Y'$.
Let $\ep_3>0$ be chosen so that an $\ep_3$-Margulis
tube cannot meet a curve of length $L_h$ which is not homotopic to a power of its
core curve. So, if a curve of length $L_h$ intersects $g(X)$ it must do so in
the image of the $\ep_3$-thick part of $X$, which consists of a bounded number
of pieces of bounded diameter. Similarly, each of $Y$ and $Y'$ is a
union of a bounded diameter neighborhood of its boundary and possibly an annulus inside
an $\ep_3$-Margulis tube. Since a uniformly bounded number of homotopy classes of primitive curves of length
at most  $L_h$ can intersect a set of uniformly bounded diameter, we
obtain the desired bound $m(L_h)$.

Now, if $\beta$ is a curve such that $\beta^M$ lies below $\gamma^M$, but $\beta^*$ lies above
$\hhat\alpha$, then either $\alpha$ does not overlap $\beta$ or 
there exists a homotopy $H_\beta$ from $\beta^M$ to $\beta^*$  through curves of length
at most $L_h$ which intersects $g(X)\union Y\union Y'$.  Since
at most three simplices of $h_W$ contain
curves which do not overlap $\alpha$, and we have the above bound
$m(L_h)$, 
we conclude that there at most $d=d(S,L_0)$ such curves.
Therefore, $\mathcal C(\hhat\alpha,L)$ contains at least $d_W(\alpha,\nu^-(\rho))-b-d$ curves.
If we also choose $D>K+b+d$ the proof is complete.
\end{proof}

\end{proof}

\section{Uniform cores}
\label{uniform cores section}

In this section we show that all of the marked hyperbolic manifolds in
$AH_0(M,P)$ can be given compact cores whose geometry is controlled by
a finite family of ``models'', in a suitable sense. These
are quite different from the bilipschitz models for entire manifolds
that are constructed in \cite{ELC1,ELC2}, in that they give a model
only for a compact subset of the manifold and not for its ends.
Because a sequence
of elements of $AH_0(M,P)$ can degenerate along annuli in
$\Sigma(M,P)$, our notion of control must allow this kind of
degeneration. Similarly the possibility of changes of marking inside
the characteristic submanifold must be taken into account. We make
this precise with the following definition.

If $(M,P)$ is a pared manifold with pared incompressible boundary,
a {\em model core} is given by a triple $(\AAA,m,\ep)$ where
$\AAA$ is a robust collection of annuli in $M$, 
$m$ is a metric on $M_\AAA$, and $\ep\in(0,\mu_3/2)$.
If $\rho\in AH_0(M,P)$ we say that a model core $(m,\AAA,\ep)$ {\em
  controls} $\rho$ if there exists
an embedding $$ f : (M,P)  \to (N_\rho\ssm\MT_{\ep}(P),\partial \MT_{\ep}(P))$$
such that
\begin{enumerate}
\item $(f(M),f(P))$ is a relative compact core for $N_\rho\setminus\MT_\ep(P)$,
\item there is a homeomorphism $\phi:(M,P)\to (M,P)$ which is the
  identity on the complement of $\Sigma(M,P)$, such that $f\circ \phi$
is in the
  homotopy class determined by $\rho$,
  \item the restriction of $f$ to $M_{\mathcal A}$ is 2-bilipschitz
    with respect to $m$,
  \item $f(M_{\mathcal A}) \subset N_{\rho} \setminus
    \MT_{2\ep}(f(\AAA))$,
  \item $f(\partial M) \subset N_\rho\setminus (N_\rho)_{(0,\ep)}$.
\end{enumerate}
Here $\MT_\ep(f(\AAA))$ denotes the union of Margulis tubes
$\MT_\ep(f(A))$ for annuli $A\in\AAA$.

We call $f$ a {\em model core map} for $(\AAA,m,\ep)$.

\begin{theorem}{uniform cores}{}
If $(M,P)$ is a pared 3-manifold with pared
incompressible boundary then, for any $\ep < \mu_3/2$, 
there exists a finite collection of model cores $(\AAA,m,\ep)$
such that every $[\rho]\in AH_0(M,P)$ is controlled by one of them.
\end{theorem}

\realfig{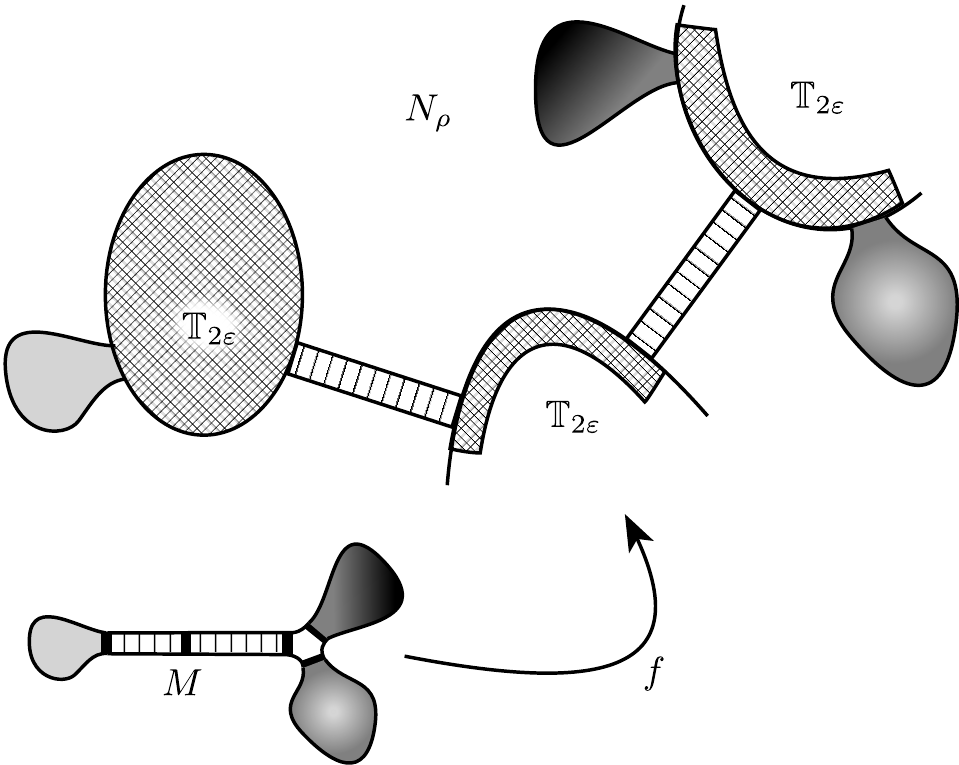}{A model core and its controlling map. Here $\AAA$
has 5 components (indicated by thick segments) whose associated solid
tori map to three Margulis tubes.}

Theorem \ref{uniform cores} will follow directly from the following statement:

\begin{theorem}{core for sequence}
If $(M,P)$ is a pared 3-manifold with pared
incompressible boundary and $\ep \in(0,\mu_3/2)$, then any sequence in $AH_0(M,P)$ has a
subsequence which is controlled by a single model core $(\AAA,m,\ep)$.
\end{theorem}
Indeed, if Theorem \ref{uniform cores} were to fail, there would
be a sequence $\{[\rho_n]\}$ in $AH_0(M,P)$ such that no two elements can be
controlled by the same model core. This would contradict Theorem
\ref{core for sequence}.

The proof of Theorem \ref{core for sequence} can be briefly sketched
as follows.  Starting with a sequence $\{[\rho_n]\}$ in $AH_0(M,P)$,
we will identify a robust system of annuli $\AAA$ which breaks $M$
into pieces on which a subsequence converges, embed these pieces into
corresponding geometric limits of the subsequence, and then push those
embedded pieces into the approximating manifolds $N_{\rho_n}$ and join them along
Margulis tubes to obtain the desired compact cores.

Corollary \ref{ourrelcompactness} gives the basic convergence result,
identifying a robust collection of annuli $\AAA$ such that some
subsequence of $\{[\rho_n]\}$ converges on the pieces of the cut-up
manifold $M_\AAA$ after appropriate remarking on the characteristic
submanifold.  Our basic embedding result is Proposition
\ref{partialcoreD} in Section \ref{partial cores}, which shows that, if we consider a collection of
pieces of $M_\AAA$ which live in the same geometric limit, then they
admit disjoint embeddings into that geometric limit.  The machinery
for these embedding theorems is a variation on results of
Anderson-Canary \cite{AC-cores} and Anderson-Canary-McCullough \cite{ACM}, and Section
\ref{cores background} contains some background material and notation for these results.

In Section \ref{uniform core proof} we put these ingredients together to finish the proof. 
Pieces of $M_\AAA$ which live in one geometric limit can be mapped,
via the approximating maps of the geometric limiting process, into the
manifolds $N_{\rho_n}$. Pieces that live in different geometric limits
are easy to embed disjointly since their distance in $N_{\rho_n}$ goes
to $\infty$. Thus we can eventually disjointly embed all the pieces of
$M_\AAA$ in $N_{\rho_n}$. (Something similar was done in the
surface-group case in Brock-Canary-Minsky \cite{ELC2}.)
Note that these embeddings in the geometric limit are  pared
manifolds, where the frontier annuli ${\rm Fr(M_\AAA)}$, as well as
the annuli of the original pared locus $P$, map to the boundaries of
their associated cusps.

The next step will be to attach the
pieces along the Margulis tubes in $N_{\rho_n}$ that approximate the cusps, using solid or thickened tori (whose
diameter can be growing as $n$ grows) and obtain
compact cores $C_n$. The identification of $M$ with $C_n$ is at this
point on the level of homotopy, and it will require an additional step
to obtain a homeomorphism that respects the subdivision into
pieces. The geometry of the geometric limits gives us the metric on
the pieces which defines the core model $m$ for the subsequence.

\subsection{Relative compact carriers}
\label{cores background}
In this section, we recall a criterion from Anderson-Canary-McCullough \cite{ACM} which guarantees
that a collection of subgroups of a Kleinian group is associated to a disjoint collection of submanifolds
of the quotient manifold.

We first  recall some terminology from Kleinian groups.
Let $\Gamma$ be a finitely generated, torsion-free, non-elementary Kleinian group. 
We say that 
$\Gamma$ is {\em quasifuchsian} if its domain of discontinuity $\Omega(\Gamma)$ has 2 components, each of which is invariant under the entire group,
and is {\em degenerate} if $\Omega(\Gamma)$ is connected and simply connected.
A {\em component subgroup} of $\Gamma$ is the stabilizer of a component of its domain of discontinuity. We say that
$\Gamma$ is a {\em generalized web group} if all of its component subgroups are quasifuchsian. (This includes the case
where the domain of discontinuity is empty.)
An {\em accidental parabolic} for $\Gamma$ is a non-peripheral curve $c$ in $\Omega(\Gamma)/\Gamma$
which is associated to a parabolic element of $\Gamma$.

We say that a $\Gamma$-invariant collection $\mathcal H$ of disjoint horoballs in $\mathbb H^3$ is an {\em invariant system of horoballs}
for a torsion-free Kleinian group $\Gamma$ if every element of $\mathcal H$ is invariant under a non-trivial parabolic subgroup of $\Gamma$ and
every non-trivial parabolic subgroup of $\Gamma$ fixes some element of $\mathcal H$. If $\epsilon\in (0,\mu_3)$,
then the set of pre-images of the non-compact components of $N_{(0,\ep)}$, where $N=\mathbb H^3/\Gamma$,
is an invariant system of horoballs for $\Gamma$. If $\hhat\Gamma\subset \Gamma$ and $\mathcal H$ is an invariant
system of horoballs for $\Gamma$, then the collection $\hhat{\mathcal H}$ of elements of $\mathcal H$ based at fixed points of
non-trivial parabolic subgroups of $\hhat\Gamma$ is an invariant system of horoballs for $\hhat\Gamma$, which we call the 
{\em induced sub-collection of invariant horoballs} for $\Gamma$.
Notice that if $\mathcal H$ is the pre-image of the non-compact components of $N_{(0,\ep)}$, then $\hhat{\mathcal H}$
need not be the pre-image of the non-compact components of $\hhat N_{(0,\ep)}$. It is this unpleasant fact which
necessitates the introduction of this cumbersome terminology. However, for most purposes one can simply
imagine that our invariant system of horoballs is associated to the non-compact components of the $\epsilon$-thin part.

If $\mathcal H$ is an invariant system of horoballs for a finitely generated, torsion-free Kleinian group $\Gamma$,
let 
$$N^\mathcal H=\left(\mathbb H^3-\bigcup_{H\in\mathcal H} H\right)/\Gamma.$$
For example, if $\mathcal H$ is the pre-image of the non-compact components of $N_{(0,\ep_0)}$, then $N^0=N^\mathcal H$.
If $(M,P)$ is a relative compact core for $N^\mathcal H$, then $\Gamma$
has an accidental parabolic if and only if there is an essential annulus in $(M,P)$ joining a geometrically
finite component of $\partial_0(M,P)$ to a component of $P$. If  every geometrically finite component of 
$\partial_0(M,P)$ is incompressible and $\Gamma$ has no accidental parabolics, then it is either a generalized web group
or degenerate (see  \cite[Lemma 3.2]{ACM}).

We will make crucial use of a criterion which guarantees that a
collection of subgroups of a Kleinian group is associated to a
collection of disjoint submanifolds.  A finitely generated subgroup
$\Theta$ of a Kleinian group $\Gamma$ is {\em precisely embedded} if
it is the stabilizer, in $\Gamma$, of its limit set and if
$\gamma\in\Gamma-\Theta$, then there is a component of
$\Omega(\Theta)$ whose closure contains
$\gamma(\Lambda(\Theta))$ (here $\Lambda$ denotes the limit set of a Kleinian group). More generally, a collection
$\{\Gamma_1,\ldots,\Gamma_n\}$ of precisely embedded subgroups of
$\Gamma$ is a {\em precisely embedded system} if whenever $\gamma\in
\Gamma$ and $i\ne j$, then there is a component of $\Omega(\Gamma_i)$
whose closure contains $\gamma(\Lambda(\Gamma_j))$.

We now give a strong pared version of what it means for a subgroup to
be associated to a submanifold of the quotient manifold.  If $\hhat\Gamma$
is a subgroup of a torsion-free, finitely generated Kleinian group
$\Gamma$ with invariant system of horoballs $\mathcal H$,
then $(Y,Z)$ is a {\em relative compact carrier} for
$\Gamma$ if  $Y\subset N^{\mathcal H}$,
$Z=Y\cap\partial  N^{\mathcal H}$,
and $(Y,Z)$ lifts to a relative compact core for $\hhat N^{\hhat{\mathcal H}}$, where
$\hhat N=\mathbb H^3/\hhat\Gamma$, 
$N=\mathbb H^3/\Gamma$ and $\hhat{\mathcal H}$ is the induced subcollection
of invariant horoballs for $\Gamma$.

The following result combines Lemma 4.1 and Proposition 4.2 in \cite{ACM}.

\begin{proposition}{relcc}
Let $\Gamma$ be a torsion-free Kleinian group with invariant system of horoballs $\mathcal H$  and 
let $\{\Gamma_1,\ldots,\Gamma_n\}$ be a collection of 
non-conjugate generalized web subgroups of $\Gamma$.
Then $\{\Gamma_1,\ldots,\Gamma_n\}$  is a precisely embedded
system if and only if there exists a disjoint collection $\{R_1,\ldots,R_n\}$ of compact 
submanifolds of $N^{\mathcal H}$
such that, for all $j$, $R_j$ is a relative compact carrier for $\Gamma_j$ and no component of $N^{\mathcal H}-R_j$ is a
compact twisted $I$-bundle whose associated $\partial I$-bundles lies in $\partial R_j$.
\end{proposition}

\subsection{Partial Cores}
\label{partial cores}

In this section we establish our basic embedding result for partially convergent sequences
in $D(M,P)$.
In our setting  we have a robust collection
$\AAA$ of essential annuli and a sequence $\{\rho_n\in D(M,P)\}$ which
is convergent up to inner automorphism on some subset of components
of $M_\AAA$ (this situation will arise, in Section \ref{uniform core
  proof}, using Corollary \ref{ourrelcompactness}). Assuming also that
the sequence has a geometric limit, and with additional assumptions
on parabolics, we will obtain a collection of
disjoint relative compact carriers in the geometric limit for the
convergent components.

The following convention will be used from now on: 
Given a component $R$ of $M_\AAA$ on which 
$\{\rho_n\}$ converges up to inner automorphism, 
there exists a sequence $\{g_n\}\subset\pi_1(M)$ such that
$\{\rho_n^{\rho_n(g_n)}|_{\pi_1(R)}\}$ converges in $D(\pi_1(R))$. 
We take $\rho^R\in D(R)$ to  be the limit of such a sequence. Notice that
$\rho^R$ is well-defined up to conjugacy in the geometric limit group.

\begin{proposition}{partialcoreD} 
  Let $(M,P)$ be a pared 3-manifold with pared incompressible boundary
  and let $\mathcal A$ be a robust collection of essential annuli in
  $(M,P)$.  Suppose that $\{\rho_n\}$ is a sequence in $D(M,P)$ which
  is convergent up to inner automorphism on a collection $\RR$ of
  components of $(M_\mathcal A,P_\mathcal A)$ so that
\begin{enumerate}
\item
$\ell_{\rho_n}(\mathcal A)\to 0$, 
\item
$\{\rho_n(\pi_1(M))\}$ converges geometrically to $\Gamma$, and
\item
if $B$ is an essential annulus in a component of $(M_\mathcal A,P_\mathcal A)$,
then \hbox{$\lim \ell_{\rho_n}(B)>0$}.
\end{enumerate}

If $\ep\in (0,\ep_0]$, then there exists a disjoint collection $\{ (Y_R,Z_R)\}_{(R,Q)\in\RR}$ of relative compact
carriers for $\{\rho^R(\pi_1(R))\}_{(R,Q)\in\RR}$ in $N^\ep$ where $N=\mathbb H^3/\Gamma$ and $N^\ep$ is
obtained from $N$ by removing the non-compact components of $N_{(0,\ep)}$.
\end{proposition}

We begin by showing that if $(R,Q), (S,T)\in\RR$, 
then any conjugate of the limit set of $\rho^R(\pi_1(R))$
intersects $\rho^S(\pi_1(S))$ in at most one point.  We next show that each
$\rho^R(\pi_1(R))$ is either a generalized web group or a degenerate group
without accidental parabolics. We can then use Proposition \ref{relcc} and the Covering Theorem 
\cite{canary-cover,thurston-notes}
to complete the proof of Proposition \ref{partialcoreD},
as in the proof of \cite[Prop 6.4]{ELC2}.

Our first lemma is a common generalization of Proposition 2.7 from \cite{ACCS} and Proposition 6.7
from \cite{ELC2}.

\begin{lemma}{onepoint}
Let $(M,P)$ be a pared 3-manifold with pared incompressible boundary
and let $\mathcal A$ be a robust collection of essential annuli in $(M,P)$.
Suppose that $\{\rho_n\}$ is a sequence in $D(M,P)$ which is convergent up to inner automorphism on a collection
$\RR$ of components of $(M_\mathcal A,P_\mathcal A)$
so that
\begin{enumerate}
\item
$\ell_{\rho_n}(\mathcal A)\to 0$, and
\item
$\{\rho_n(\pi_1(M))\}$ converges geometrically to $\Gamma$.
\end{enumerate}

If  (i) $R$ and $S$ are distinct elements of $\RR$  and $\gamma \in \Gamma$ or
(ii) $R=S$ and \hbox{$\gamma\notin \rho^R(\pi_1(R))$}, then
the intersection of limit sets
$$\Lambda(\gamma \rho^R(\pi_1(R)) \gamma^{-1}) \intersect
\Lambda(\rho^S(\pi_1(S)))$$ 
contains at most one point.
\end{lemma}

\begin{proof}
A result of Anderson \cite{anderson-intersection} and Soma \cite{soma-intersection} implies that
if $\Phi_1$ and $\Phi_2$ are
non-elementary, finitely generated subgroups of a torsion-free Kleinian group $\Gamma$, then
$$ \Lambda(\Phi_1)\cap \Lambda(\Phi_2) =
\Lambda(\Phi_1\cap\Phi_2)\union P(\Phi_1,\Phi_2)$$
where $P(\Phi_1,\Phi_2)$ is the set of
points $x\in\Lambda(\Gamma)$ such that the stabilizers of $x$
in $\Phi_1$ and $\Phi_2$ are rank one parabolic subgroups which
generate a rank two parabolic subgroup of $\Gamma$.
(Anderson and Soma's results require that the subgroups are topologically tame, but it is now
known, by work of Agol \cite{agol} and Calegari-Gabai \cite{calegari-gabai}, 
that all finitely generated, torsion-free Kleinian groups are topologically tame.)

Without loss of generality we can assume that $\{\rho_n|_{\pi_1(R)}\}$ converges to $\rho^R$ and
there exists $g_n\in \pi_1(M)$ so that $\{\rho_n^{\rho_n(g_n)}|_{\pi_1(S)}\}$ converges to $\rho^S$.

One may show that $\gamma\rho^R(\pi_1(R))\gamma^{-1}\cap\rho^S(\pi_1(S))$ is purely parabolic, i.e.
consists entirely of parabolic and trivial elements,
exactly as in the proofs of  \cite[Lemma 2.4]{ACCS} and \cite[Lemma 6.7]{ELC2}.
Therefore,
$\Lambda(\gamma\rho^R(\pi_1(R))\gamma^{-1}\cap\rho^S(\pi_1(S)))$ contains at most one point.
Thus, it only remains to prove
that 
$$P(\gamma\rho^R(\pi_1(R))\gamma^{-1},\rho^S(\pi_1(S))) = \emptyset.$$

\medskip

Suppose that $x\in P(\gamma\rho^R(\pi_1(R))\gamma^{-1},\rho^S(\pi_1(S)))$.
The stabilizer $\stab_{\rho^S(\pi_1(S))}(x)$ is generated by
$\rho^S(s)$, and
$\stab_{\gamma\rho^R(\pi_1(R))\gamma^{-1}}(x)$ is generated by
$\gamma\rho^R(r)\gamma^{-1}$ where $r$ and $s$ are primitive elements
of $\pi_1(R)$ and $\pi_1(S)$, respectively.
Since these two elements must commute, 
$h_n r h_n^{-1}$ commutes with
$g_nsg_n^{-1}$ for sufficiently large $n$, see \cite[Prop. 2.7]{ELC2}.

If $h_n r h_n^{-1}$ and
$g_nsg_n^{-1}$ lie in a cyclic subgroup of $\pi_1(M)$, then, since they are primitive,
they agree, up to taking inverses, so their limits cannot generate a rank two abelian group.
Otherwise, $h_n r h_n^{-1}$ and $g_nsg_n^{-1}$ are primitive elements generating a 
rank two abelian subgroup $\Delta_n$ of $\pi_1(M)$ for all sufficiently large $n$.
This can only occur if $R=S$,
and the rank two abelian subgroup $\Delta_n$ is the maximal rank two abelian subgroup $\Delta$ of $\pi_1(R)$ containing $s$.
So, $\gamma$ lies in the geometric limit of $\{\rho_n(\Delta)\}$, which is simply $\rho^R(\Delta)$
(since any element of the geometric limit of $\{\rho_n(\Delta)\}$ is a parabolic element which shares
a fixed point with $\rho^R(\Delta)$, and hence, since $\rho^R(\Delta)$ has rank two and the geometric limit
is discrete, has a power which lies in $\rho^R(\Delta)$, which is again ruled out by an application of \cite[Prop. 2.7]{ELC2}).
We have achieved a contradiction, so $P(\gamma\rho^R(\pi_1(R))\gamma^{-1},\rho^S(\pi_1(S)))$ is empty, 
which completes the proof.
\end{proof}

The following lemma is the main new ingredient in our proof.

\begin{lemma}{webordegenerate}
If we are in the setting of Proposition \ref{partialcoreD} and $(R,Q)$ is a component of $\mathcal R$, then
$\rho^R(\pi_1(R))$ is either a generalized web group
or a degenerate group without accidental parabolics. 
\end{lemma}

\begin{proof}
Let $\mathcal H$ be the invariant system of horoballs for $\Gamma$ obtained by considering
the pre-images of the non-compact components of $N_{(0,\ep)}$ where $N=\mathbb H^3/\Gamma$.
If \hbox{$(R,Q)\in\RR$}, let $\mathcal H_R$ be the induced subcollection of $\mathcal H$ associated
to the subgroup $\rho^R(\pi_1(R))$ and let $N_R^\mathcal H=N_{\rho^R}^{\mathcal H_R}$.

We first show that a relative compact core $(M_R,P_R)$ of $N_R^\mathcal H$ has pared incompressible boundary.
Recall that  Lemma \ref{pieces are pared} implies that $(R,Q)$ is a pared manifold with pared incompressible boundary.
Let $P_R^Q$ denote the collection of components of $P_R$ whose fundamental groups are in the conjugacy class
of $\rho^R(\pi_1(Q_0))$ for some component $Q_0$  of $Q$.
Since there is a pared homotopy
equivalence  
$$j:(R,Q)\to (M_R,P_R^Q),$$
every component of $\partial_0(M_R,P^Q_R)$ is incompressible
(see Bonahon \cite[Prop. 1.2]{bonahon} or \cite[Lemma 5.2.1]{canary-mccullough}).
Since $\partial_0(M_R,P_R)$ is obtained from $\partial_0(M_R,P_R^Q)$ by removing incompressible annuli,
$\partial_0(M_R,P_R)$ is also incompressible, so $(M_R,P_R)$ also has pared incompressible boundary.

We next show that  $\rho^R(\pi_1(R))$  has no accidental parabolics.  
If $\rho^R(\pi_1(R)))$ has an accidental parabolic, there exists an essential annulus $E$ in $(M_R,P_R)$ 
so that one component of $\partial E$ is contained in $P_R$ and the other is contained in $\partial_0(M_R,P_R)$.
Notice that $E$ is also an essential annulus in $(M_R,P_R^Q)$.
Thus, there is a component $\Sigma_0$ of the characteristic submanifold $\Sigma(X_R,P^Q_R)$, so
that $E$ is isotopic into $\Sigma_0$. Johannson's Classification Theorem \cite[Thm. 24.2]{johannson} implies
that we may assume that the pared homotopy equivalence $j$ between $(R,Q)$ and $(M_R,P_R^Q)$ has the
property that
$j(\Sigma(R,Q))=\Sigma(M_R,P_R^Q)$ and 
$j$ is a homeomorphism from the (closure of the) complement of $\Sigma(R,Q)$ to the
(closure of the) complement of $\Sigma(M_R,P_R^Q)$.
If $\Sigma_0$ is a solid or thickened torus component of $\Sigma(M_R,P_R^Q)$, then $\Sigma_1=j^{-1}(\Sigma_0)$ is a solid
or thickened torus component of $\Sigma(R,Q)$. It follows that any annulus $B$ in the frontier of $\Sigma_1$ would be
an essential annulus in $(R,Q)$ with the property that $\ell_{\rho^R}(B)=0$,
which is disallowed by our assumptions.
If $\Sigma_0$ is an interval bundle component of $\Sigma(M_R,P_R^Q)$, 
then $\Sigma_1=j^{-1}(\Sigma_0)$ is an interval bundle component of $\Sigma(R,Q)$.
However, by \cite[Lemma 2.11.3]{canary-mccullough}, the restriction of $j$ to $(\Sigma_1,{\rm Fr}(\Sigma_1))$
is pared homotopic to a pared homeomorphism $h:(\Sigma_1,{\rm Fr}(\Sigma_1))\to (\Sigma_0,{\rm Fr}(\Sigma_0))$.
Then, $B=h^{-1}(E)$ would  again be an essential annulus in $(R,Q)$
with the property that $\ell_{\rho^R}(B)=0$,
which is again disallowed by our assumptions. Therefore, $\rho^R(\pi_1(R))$ has no accidental parabolics.

Since the relative compact core of $N_R^\mathcal H$ has pared incompressible boundary and
$\rho^R(\pi_1(R))$ has no accidental parabolics, Lemma 3.2 in \cite{ACM} implies that
$\rho^R(\pi_1(R))$ is either a generalized web group or degenerate group without accidental parabolics.
\end{proof}

\medskip\noindent
{\em Proof of Proposition \ref{partialcoreD}:}
Let $\{\Gamma^1,\ldots,\Gamma^s\}=\{\rho^R(\pi_1(R))\}_{(R,Q)\in\RR}$.
We re-order so that $\{\Gamma^1,\ldots,\Gamma^r\}$ are generalized web groups and
 $\{\Gamma^{r+1},\ldots,\Gamma^s\}$ are degenerate groups without accidental parabolics.
Let $(R_i,Q_i)$ be the component of $\RR$, so that $\Gamma_i=\rho^{R_i}(\pi_1(R_i))$.

For $i,j\in\{1,\ldots,r\}$, 
Lemma \ref{onepoint} implies that if $\gamma\in\Gamma$  and $i\ne j$
or if $i=j$ 
and $\gamma \in \Gamma-\Gamma^i$, then the limit sets
of $\gamma\Gamma^i\gamma^{-1}$ and $\Gamma^j$ intersect in
at most one point. Since each component of $\Omega(\Gamma^i)$ is bounded by a quasi-circle, we 
immediately conclude that $\gamma(\Lambda(\Gamma^j))$ is contained in the closure of a component
of $\Omega(\Gamma^i)$. Therefore,
$\{\Gamma^1,\ldots,\Gamma^r\}$ is a precisely embedded collection of
generalized web subgroups of $\Gamma$. 
Proposition \ref{relcc} then implies that there is an associated collection
$\{(Y_1,Z_1),\ldots, (Y^r,Z_r)\}$ of disjoint relative compact carriers for $\{\Gamma^1,\ldots,\Gamma^r\}$
in $N^{\mathcal H}$.  

Since $\Gamma^{r+1}$ is a degenerate group without accidental parabolics,
the Covering Theorem \cite{canary-cover,thurston-notes} may be used, exactly as in the proof of 
\cite[Prop. 6.10]{ELC2}, to show that there exists a neighborhood $U$ of
an end of $N_{R_{r+1}}^\mathcal H\cong F_{r+1}\times\mathbb R$, for some compact surface $F_{r+1}$,
which is identified with
$F_{r+1}\times (k_{r+1},\infty)$ and $\pi_{R_{r+1}}(U)$ embeds in $N$, under the obvious
covering map $p_{r+1}:N^\mathcal H_{R_{r+1}}\to N$. One may then choose 
a relative compact carrier $(Y_{r+1},Z_{r+1})$  for $\Gamma_{r+1}$ 
of the form
$$(Y_{r+1},Z_{r+1})=(p_{r+1}(F_{r+1}\times [s,s+1]),p_{r+1}(\partial F_r\times [s,s+1]))$$ 
for some $s>k_{r+1}$ so that $Y_{r+1}$ is disjoint from
$Y_1\cup\cdots \cup Y_r$. One similarly  uses the Covering
Theorem to iteratively choose a relative compact carrier $(Y_{r+j},Z_{r+j})$ for $\Gamma_{r+j}$ which is  
disjoint from $Y_1\cup\cdots\cup Y_{r+j-1}$.
One finally arrives at a disjoint collection $\{ (Y_R,Z_R)\}_{(R,Q)\in\RR}$ of
relative compact carriers for $\{\Gamma^i\}_{i=1}^s$. 
\qed

\subsection{Proof of the uniform cores theorem}
\label{uniform core proof}
We now complete the proof of Theorem \ref{core for sequence} (and hence
Theorem \ref{uniform cores}). The proof breaks into several steps
which we briefly summarize.

The first step is to use the convergence theorems we have been
developing to identify a robust system of annuli $\AAA$ so that (a
subsequence of) our
given sequence $\{\rho_n\}$ converges on the pieces of $M_\AAA$. We then
find an embedded submanifold $Y$ in the union of geometric
limits of $\{N_{\rho_n}\}$, and a homotopy equivalence $h:M_\AAA \to Y$.
In Step 1 we do this for each geometric limit separately. In Step 2 we
combine the geometric limits and produce an
embedding $\psi_n:Y \to N_{\rho_n}$ for each large enough
$n$. Attaching the pieces of $Y_n  = \psi_n(Y)$ along Margulis tubes,
we obtain the submanifold $C_n\subset N_{\rho_n}$ which will be our
compact core. In Step 3, we combine $h$ and $\psi_n$ to get a map $s_n
: M\to C_n$ in the correct homotopy class.

Note that, at this point,
$s_n$ is not known to be a homeomorphism even on the components of
$M_\AAA$. Indeed, even in the case of a convergent sequence of
representations the homeomorphism type of the limit may differ from
that of the approximates, as discovered by Anderson-Canary \cite{AC-pages}, and this
phenomenon accounts for many of the difficulties in this proof.

In Step 4 we show that $s_n$ is a homotopy equivalence and conclude
that $C_n$ is in fact a (relative) compact core of $N_{\rho_n}$. In
Step 5 we show that $s_n$ is indeed homotopic to a homeomorphism $\hhat
s_n$. Note that here the assumption that $[\rho_n]\in AH_0(M,P)$ is
crucial: it means that an embedding from $M$ to $N_{\rho_n}$ in the
right homotopy class actually exists.  It requires some extra work to
ensure that $\hhat s_n$ respects the decomposition induced by the
annuli $\AAA$. Finally in Step 6 we use the geometry of the geometric
limits to choose a metric $m$ on $M_\AAA$ which makes our final maps $2$-bilipschitz.

\subsubsection*{Step 1: Embedding the partial cores}
Consider a sequence $\{\rho_n\}$ in $D(M,P)$ of representatives of the
sequence $\{[\rho_n]\}$. 
We may apply Corollary \ref{ourrelcompactness} to remark, subdivide by annuli and
obtain a subsequence that converges on the pieces. That is, 
let  $\{\phi_n:(M,P)\to (M,P)\}$ be the sequence of
homeomorphisms which are the
identity on the complement of $\Sigma(M,P)$, and $\mathcal A$ the robust collection of annuli
for which the conclusions of Corollary  \ref{ourrelcompactness} hold.
To lighten the notation we can assume, without loss of generality,
that each $\phi_n$ is the identity  and that $\{\rho_n\}$ is the subsequence. Hence we have
\begin{enumerate}
\item
$\lim\ell_{\rho_n}(\AAA)=0$, 
\item 
$\{\rho_n\}$ converges on $M_\mathcal A$ up to conjugacy, and
\item 
if $B$ is an essential annulus in  $(M_\AAA,P_\AAA)$, then 
$$\lim\ell_{\rho_n}(B)>0.$$
\end{enumerate}

For each component $(R,Q)$ of $M_\mathcal A$ we have a sequence $\{\rho_n^{g_n}\}$ of conjugates
which converge on $\pi_1(R)$ to $\rho^R$, and a corresponding
sequence  $\{b^R_n \in N_{\rho_n}\}$ of basepoints,  which are the projections of a fixed basepoint $\zero\in\mathbb H^3$
as in the setup of Lemma \ref{basepoints and limits}. 
We may pass to a subsequence so that if $(R,Q)$ and $(S,T)$ are components of
$M_\mathcal A$, then $d(R,S)=\lim d(b_n^R,b_n^S)$ 
exists and lives in $[0,\infty]$. We say two components $(R,Q)$ and $(S,T)$ are
{\em nearby} if $d(R,S)$ is finite. 
This defines an equivalence relation on the set of components of
$M_\mathcal A$. For an equivalence class $\RR$ of components, Lemma
\ref{basepoints and limits} implies that there is a single sequence of
conjugates $\{\rho_n^{g_n}\}$ which converges up to inner automorphisms on
$\RR$, and (again passing to a subsequence) shares a geometric limit
$\Gamma^\RR$ with quotient $N_\RR$.

We can then apply Proposition \ref{partialcoreD} to
obtain, for each equivalence class $\RR$,  a collection $\{(Y_R,Z_R)\}_{(R,Q)\in\RR}$ of disjoint relative
compact carriers for $\{\rho^R(\pi_1(R))\}_{(R,Q)\in\RR}$ in
$N^{2\ep}_\RR$.
Let $Y_\RR=\bigcup_{(R,Q)\in\RR} Y_R$ and $Z_\RR=\bigcup_{(R,Q)\in \RR} Z_R$.

For each $(R,Q)\in\RR$, let $Z_R^Q$ be the collection of components of $Z_R$ which lie in
the homotopy class of $\rho^R(\pi_1(Q))$. There exists a map of pairs
$$h_R:(R,Q)\to (Y_R,Z_R^Q)$$ 
which is a homotopy equivalence from $R$ to $Y_R$ and restricts to an orientation-preserving embedding on $Q$.

The following elementary lemma from \cite{canary-mccullough} implies that
$h_R$ is a homotopy equivalence of pairs.

\begin{lemma}{paredupgrade}{\rm (Canary-McCullough \cite[Prop. 5.2.3]{canary-mccullough})}
If $(M_1,P_1)$ and $(M_2,P_2)$ are  pared manifolds and
$h:(M_1,P_1)\to (M_2,P_2)$ is a map of pairs, so that $h$ is a homotopy equivalence from
$M_1$ to $M_2$, then $h$ is a pared homotopy equivalence.
\end{lemma}

We may combine the maps on each component to obtain a map
$$h_\RR=\bigcup h_R:\bigcup_{(R,Q)\in \RR} (R,Q)\to (Y_\RR,Z_\RR).$$

\subsubsection*{Step 2: Building the core $C_n$}
A crucial step in the proof is the construction of a compact core
$C_n$ in the approximate $N_{\rho_n}$ for 
all large enough $n$.  Using elementary facts about geometric limits
we will pull back the carriers $Y_\RR$ to submanifolds of
$N_{\rho_n}$, and then attach them along the thin parts associated to
the annuli $\AAA$ to obtain $C_n$ together with maps from $M$ to $C_n$ which
we will then show are homotopy equivalences, and eventually promote to homeomorphisms.

We first make a small adjustment to $(Y_\RR,Z_\RR)$,
caused by a technical need to assure that $C_n$ may be pared so that
it is a relative compact core for $N_{\rho_n}\ssm\MT_{\ep,n}(P)$,
whereas the annuli of $Z_\RR$ have been chosen to lie in
$\boundary\MT_{2\ep}(Z_\RR)$. Thus, for each annulus $P_0$ of $P$
which lies in a component $(R,Q)$ of $M_\AAA$, let $Z_0 = h_R(P_0)$ be the
associated component of $Z_R$. We ``push'' $Z_0$ into
$\boundary \MT_\ep(Z_0)$ by 
adding to $Y_R$ a radial collar
of $Z_0$ in $\MT_{2\ep}(Z_0)\ssm \MT_\ep(Z_0)$. To
simplify notation we continue to call the resulting pared manifolds
$(Y_R,Z_R)$, and the homotopy equivalences $h_R$.

For all large enough $n$, Lemma \ref{basic geometric limit facts} provides a 2-bilipschitz map
$$\psi_n^\RR: Y_\RR\to N_{\rho_n}$$
such that for each $(R,Q)\in\RR$,
\begin{enumerate}
\item 
$\psi_n|_{ Y_R}$ is in the homotopy class determined by
$\rho_n\circ (h_R)_*^{-1}$,
\item 
If $Z_0=h_R(P_0)$ and $P_0$ is a component of $Q\intersect P$, then
$$\psi_n(Y_R)\subset N_{\rho_n}\ssm\MT_{\ep,n}(P_0)\qquad {\rm and}\qquad \psi_n(Z_0)\subset \partial \MT_{\ep,n}(P_0).$$
\item 
If $Z_0=h_R(Q_0)$  and $Q_0$ is a component of $Q\ssm P$, then
$$\psi_n(Y_R)\subset N_{\rho_n}\ssm\MT_{2\ep,n}(Q_0)\qquad {\rm and}\qquad 
\psi_n(Z_0)\subset \partial \MT_{2\ep,n}(Q_0).$$
\end{enumerate}
(Here, $\MT_{\ep,n}(P_0)$ denotes the component of $(N_{\rho_n})_{(0,\ep)}$ in the homotopy class of $P_0$.)

 If $\RR$ and $\RR'$ are distinct equivalence classes of components of
$M_\mathcal A$, the fact that $d(R,R')=\infty$ for all $(R,Q)\in\RR$ and $(R',Q')\in\RR'$
implies that $\psi^\RR_n(Y_\RR)$ and
$\psi^{\RR'}_n(Y_{\RR'})$ are disjoint for all large enough values of
$n$.  To see this, note first that Lemma \ref{basepoints and limits}
tells us that the basepoints $b_n^R$ for $(R,Q)\in\RR$ remain a bounded
distance from the basepoints $b_n^\RR$ for the common sequence of
conjugates that converges for all $R\in\RR$.
Now the
maps $\psi_n^\RR$ take the limiting basepoint $b^\RR$ to $b_n^\RR$ 
(see  Lemma \ref{basic geometric limit facts}) 
so we may conclude that $\psi_n^\RR(Y_\RR)$ remain at bounded distance
from $b_n^R$ for each $(R,Q)\in\RR$, and similarly for $\RR'$. Hence
the distance between $\psi_n^\RR(Y_\RR)$ and $\psi_n^{\RR'}(Y_{\RR'})$
goes to $\infty$, so they are eventually disjoint.
Now let 
$$Y=\bigcup Y_\RR, \qquad h=\bigcup h_\RR, \qquad
\psi_n=\bigcup \psi_n^\RR,\qquad \textrm{and}\qquad
Y_n=\psi_n( Y),$$
noting that $\psi_n:Y\to  Y_n$ is an embedding for
large $n$. 

We now construct a submanifold $C_n$ as the union of $Y_n $ with solid 
or thickened tori in $\MT_{2\ep,n}(\AAA)-\MT_{\ep,n}(\AAA)$, each of which is attached to one or
more annuli in $\psi_n(h({\rm Fr}(M_\mathcal A)))$.
Let $\MT_{2\ep,n}(A_i)$ be a Margulis tube for a component
(possibly several) $A_i$ of $\AAA$. 
If $\ell_{\rho_n}(A_i)>0$, we append  $\bbar\MT_{2\ep,n}(A_i)$  to $Y_n$.
If  not, we choose an annulus (if $\MT_{2\ep,n}(A_i)$ is a rank one cusp)  or torus 
(if $\MT_{2\ep,n}(A_i)$ is a rank two cusp) $B_i$ in $\partial\MT_{2\ep,n}(A_i)$
which contains $ Y_n\cap \boundary\MT_{\ep,n}(A_i)$
in its interior, and append to $Y_n$ the set of  points in
\hbox{$\MT_{2\ep,n}(A_i)\ssm\MT_{\ep,n}(A_i)$} which project radially to $B_i$.

To give $C_n$ a pared structure, note first that each annulus
$P_0\subset P$ in a component $(R,Q)$ of $M_\AAA$ is already taken by
$h_R$ to an annulus in $\boundary\MT_{\ep,n}(P_0)$.
If $P_0$ is a component of $P$ lying
in a component $U$ of $M\setminus M_\AAA$, then $U$ corresponds to a
cusp $\MT_{2\ep,n}(U)$ and the intersection of
$C_n$ with $\boundary\MT_{\ep,n}(U)$ is an annulus or torus.
Putting all these components together we obtain a 
pared locus $P_n =  C_n \intersect \boundary\MT_{\ep,n}(P)$.

\subsubsection*{Step 3: Mapping $M$ to $C_n$}
Next we define a map $s_n : (M,P) \to (C_n,P_n)$, such that
\begin{itemize}
\item $s_n$ is in the homotopy class of $\rho_n$, viewed as a map from
  $M$ to $N_{\rho_n}$,
\item $s_n|_{M_\mathcal A} = \psi_n\circ h$,
\item $s_n$ maps $P$ to $P_n$ homeomorphically.

\end{itemize}

Since $[\rho_n]\in AH_0(M,P)$, there exists
an orientation-preserving embedding
$$\kappa_n:(M,P)\to (N_{\rho_n}^\ep,\partial N_{\rho_n}^\ep)$$
in the homotopy class of $\rho_n$. After an isotopy along the boundary,
we may assume that $\kappa_n(P)$ is the pared locus $P_n$ of $C_n$.

Since $\psi_n\circ h$ is homotopic to $\kappa_n|_{M_\mathcal A}$, and
each element of $\mathcal A \union P$ is
collared in $M$, there exists an extension $s'_n$ of $\psi_n\circ h$ to all
of $M$, which is homotopic to $\kappa_n$ as a map of pairs and equals
$\kappa_n$ on $P$.
If $U$ is a component of
$M\setminus M_\mathcal A$ then it is a  solid or thickened torus, and 
${\rm Fr}(U)$ is mapped by $s_n'$
into the boundary of the associated Margulis tube $\MT_{2\ep,n}(U)$.
Since $\pi_1(U)$ is a maximal abelian subgroup of $\pi_1(M)$ and $\pi_1(U)$ is its own
centralizer in $\pi_1(M)$,
the restriction $s'_n|_U$ is homotopic, relative to ${\rm Fr}(U)\union(\boundary U\intersect P)$, into 
$C_n \intersect \MT_{2\ep,n}(U)$.
After performing such a homotopy for
each $U$ we obtain the desired map $s_n$.

\subsubsection*{Step 4: Showing $C_n$ is a core}
We now claim that $(C_n,P_n)$ is a relative compact core for $N_n-\MT_{\ep,n}(P)$. 
To prove this we first note the
following lemma proved in \cite{ACM}.

\begin{lemma}{acmtoplemma2}{}
{\rm (\cite[Lemma 5.2]{ACM})} 
For $i=1,2$, let $M_i$ be a compact, orientable, irreducible
3-manifold with incompressible boundary 
and let $V_i$ be a 3-dimensional submanifold whose frontier ${\rm Fr}(V_i)$
is non-empty and incompressible. 
If $g:M_1\to M_2$ is a continuous map such that
\begin{enumerate}
\item
$g^{-1}(V_2)=V_1$,
\item
$g$ restricts to a homeomorphism from $V_1$ to $V_2$,
and
\item
$g$ restricts to a homotopy equivalence from $M_1\ssm V_1$ to
 $M_2\ssm V_2$,
\end{enumerate}
then $g$ is a homotopy equivalence.
\end{lemma}

We can apply this lemma with $M_1=M$ and $M_2=C_n$, 
where $V_1$ is a collar neighborhood $\NN(\AAA)$. Since $s_n$ by
construction is an embedding on $\AAA$ and takes $M\ssm\AAA$ to the
complement of $s_n(\AAA)$,  we can assume after a small homotopy that
$s_n$ is an embedding on $V_1$, and setting $V_2 = s_n(V_1)$, that
$s_n^{-1}(V_2) = V_1$ as well. The components of $M\ssm \AAA$ are
isotopic to components of $M_\AAA$ or to the solid or thickened tori
of $M\ssm M_\AAA$, and $s_n$ is a homotopy equivalence from each of
these to the corresponding component of $C_n \ssm s_n(\AAA)$, again by
construction. Moreover we note that $s_n$ is bijective on the
components of $M\ssm \AAA$ since no two components have conjugate
fundamental groups, and that it is surjective to the components of
$C_n\ssm s_n(\AAA)$ by definition of $C_n$. We conclude that $s_n$ is a
homotopy equivalence of $M_1\ssm V_1$ to $M_2\ssm V_2$. Lemma
\ref{acmtoplemma2} therefore implies that $s_n:M\to C_n$ is a homotopy
equivalence.

Since $s_n$ is in the homotopy class of $\rho_n$, as a map into $N_{\rho_n}$, the inclusion of $\pi_1(C_n)$
into $\pi_1(N_{\rho_n})$ is an isomorphism, so $C_n$ is a compact core for $N_{\rho_n}$.
Then, by construction, $(C_n,P_n)$ is a relative compact core for $N_{\rho_n}^\ep$.
Lemma \ref{paredupgrade} implies that
$$s_n:(M,P)\to (C_n,P_n)$$
is a homotopy equivalence of pairs.
It follows from this that $(C_n,P_n)$ is a relative compact core for
$N_{\rho_n} \ssm \MT_{\ep,n}(P)$.

\subsubsection*{Step 5: Homotopy to the final embedding}
We next show that \hbox{$s_n:(M,P)\to (C_n,P_n)$} is pared homotopic to a homeomorphism.
Recall that relative compact cores are unique up to admissible isotopy and
their complements have  product structures. (See section \ref{deformation}.)
Therefore, there exists a pared homeomorphism
\hbox{$g_n: (C_n,P_n)\to(\kappa_n(M),P_n)$} in the isotopy class of the inclusion map.
Then $\hhat s_n=g_n^{-1}\circ\kappa_n$ is a pared homeomorphism in the homotopy class of $\rho_n$
and hence homotopic to $s_n$.
Since 
$N_{\rho_n}^\ep$ deformation retracts onto $C_n$,
the homotopy  between $s_n$ and $\hhat s_n$ can be presumed to remain in $C_n$.

We wish to show that 
$\hhat s_n$ can be admissibly isotoped so that it preserves the
decomposition of $\AAA$. That is, it takes 
any component $R$ of $M_\mathcal A$ to $\psi_n( Y_R)$ and 
any component $U$ of $M-M_\mathcal A$ to $\MT_{2\ep,n}(U)\cap C_n$. 
To do this, we first show that $\hhat s_n$ can be admissibly isotoped so that 
$\hhat s_n({\rm Fr}(M_\mathcal A))=s_n({\rm Fr}(M_\mathcal A))$.

Let
${\mathcal B }$ denote the collection of components of  ${\rm Fr}(M_\mathcal A)$.
If $B\in \mathcal B$, let $\hhat B=\hhat s_n^{-1}(s_n(B))$ 
and 
$$\hhat{\mathcal B}=\{\hhat B\}_{B\in\mathcal B}.$$
Notice that if $B\in\mathcal B$, then $s_n(B)$ is an essential annulus in $(C_n,P_n)$, since otherwise it would bound
a solid torus $V$ in $C_n$ such that $\partial V-B\subset \partial C_n$.
So,  since $\hhat s_n$ is a pared homeomorphism, $\hhat B$ is an essential annulus in $(M,P)$.
The enclosing property for characteristic submanifolds, see Johannson \cite[Prop. 10.7]{johannson},
implies that $\hhat{\mathcal B}$ is admissibly isotopic into $\Sigma(M,P)$, so we may assume that $\hhat{\mathcal B}$
is contained in the interior of $\Sigma(M,P)$.

If a component $B$ of $\hhat\BB$ is homotopic into a solid or
thickened torus component $X$ of $\Sigma(M,P)$ but is not contained in
$X$, it must be contained in an adjacent interval bundle component,
and cobound a solid torus with a component of ${\rm Fr}(X)$. Since
$\hhat\BB$ is embedded, all these solid tori are disjoint or nested.
Thus, after adjusting $\hhat s_n$ by an isotopy supported on a regular
neighborhood of the union of
such solid tori, we may assume that each component of $\hhat\BB$ that
is homotopic to a solid or thickened torus component of $\Sigma(M,P)$
is already contained in it.

\par\medskip\noindent{\em Solid/thickened torus components:}
Consider a component $V$ of  $M\ssm M_\AAA$
which is a regular neighborhood of a solid or thickened torus component
$X$ of  $\Sigma(M,P)$.  Let ${\mathcal B}_V\subset\BB$ be the components of ${\rm Fr}(V)$ and let
$\hhat{\mathcal B}_V$ denote the corresponding subset of
$\hhat{\mathcal B}$. By the previous paragraph we know that the
components of $\BB_V$  lie in the  interior of $X$.

The pair $(X,{\rm Fr}(X))$ admits a Seifert-fibration over
$(E,d)$, where $E$ is either a disk with $d$ a collection of arcs in
$\boundary E$, or an annulus with $d$ a collection of arcs in one
component of $\boundary E$. 
There is at most one singular point in $E$ if
it is a disk, and none if it is an annulus.
Since $\hhat{\mathcal B}_V$ is a collection of essential annuli in $X$, it may be isotoped, by an isotopy supported on $X$,
to the $S^1$-bundle over a collection $e$ of arcs in $E$  with the end
points of each arc in distinct components of the complement of $d$
(see Johannson \cite[Prop. 5.6]{johannson}).
Since $s_n(\BB_V)$ bounds a solid or thickened torus in $C_n$
whose fundamental group is a maximal abelian subgroup, the same is
true for $\hhat\BB_V$. It follows that the arcs $e$ are the boundary
in $E$ of a connected region $W$ which is either an essential subannulus when $E$
is an annulus, or a disk when $E$ is a disk. Moreover if $E$ has a
singular point then it must lie in $W$, since otherwise the preimage
of $W$ would be a solid torus whose core is homotopic to a proper power of the core of $X$,
and hence generates a non-maximal subgroup of $\pi_1(M)$.
Since the cardinalities of $e$ and $d$ are equal,
it follows that each component of $e$ must be parallel to a component
of $d$ across a region that does not contain a singular point. 
It follows that $\hhat{\mathcal B}_V$ is
isotopic to the frontier of $X$, and hence to ${\mathcal B}_V$, by an isotopy supported on $V$.

\par\medskip\noindent{\em Interval bundle components:}
If $\Sigma$ is an interval bundle component of $\Sigma(M,P)$ with base
surface $F$, let $\hhat\Sigma$ be obtained from $\Sigma$ by appending the closure of ${\mathcal N}(A)$ whenever $A$ is a component
of ${\rm Fr}(\Sigma)$ contained in $\mathcal A$. Then $\hhat\Sigma$ is an interval bundle with base
surface $\hhat F$ (which is obtained from $F$ by appending collar neighborhoods of certain components of
$\partial F$).
Let $\BB_\Sigma$ denote the components of $\BB$ which are
contained in $\hhat\Sigma$ and not isotopic into a solid or thickened torus
component of $\Sigma(M,P)$. By construction, the corresponding subset
$\hhat\BB_\Sigma$ of $\hhat \BB$ is contained in the interior of
$\Sigma$.   In an interval bundle, every system of disjoint essential annuli is admissibly isotopic
to the subbundle over a multicurve in the base (again see Johannson \cite[Prop. 5.6]{johannson}), and two such multicurves are
isotopic if and only if they are homotopic.
Since  $\BB_\Sigma$ and $\hhat\BB_\Sigma$ are homotopic, they are isotopic to  sub-bundles over homotopic
collections of curves, so they are
isotopic. Therefore, we may admissibly isotope $\hhat s_n$,
by an isotopy supported on a regular neighborhood of $\hhat \Sigma$, so that 
$\hhat {\mathcal B}_\Sigma= {\mathcal B}_\Sigma$.

\medskip

Since the supports of all the resulting isotopies may be chosen to be disjoint, they can be performed simultaneously.
Therefore, we have isotoped $\hhat s_n$ 
so that $\hhat s_n({\mathcal B})=s_n({\mathcal B})$.

Since $\hhat s_n$ is a homeomorphism, $\hhat s_n({\rm Fr}(M_{\mathcal A}))=s_n({\rm Fr}(M_{\mathcal A}))$
and $s_n$ takes components of
$M\setminus{\rm Fr}(M_{\mathcal A})$ to components of $C_n\setminus s_n({\rm Fr}(M_{\mathcal A}))$,
we see that $\hhat s_n$ also takes components of
$M\setminus{\rm Fr}(M_{\mathcal A})$ homeomorphically to components of $C_n\setminus s_n({\rm Fr}(M_{\mathcal A}))$.
Notice that every component of $M_{\mathcal A}$ has non-abelian fundamental group, while
every component of $M-M_{\mathcal A}$ has abelian fundamental group.
Since $\hhat s_n$ is homotopic to $s_n$,  and no two components of $M_{\mathcal A}$ are homotopic,
we see that  $\hhat s_n(R)=s_n(R)=\psi_n( Y_R)$ if 
$R$ is a component of $M_\mathcal A$.
Similarly, if $U$ is a component of $M\ssm M_\mathcal A$, then
$\hhat s_n(U)=s_n(U)=C_n\cap \MT_{2\ep,n}(U)$. In short, our decomposition is preserved by $\hhat s_n$.

\subsubsection*{Step 6: The metric}
Let $n_0$ be large enough that our constructions work and let 
$$g:(M_\mathcal A,P_\mathcal A)\to (Y, \partial\MT_{\ep}(P_\mathcal A\cap P)\cup \partial\MT_{2\ep}(P_\mathcal A\setminus P))$$
be the diffeomorphism given by  the restriction of $\psi_{n_0}^{-1}\circ\hhat s_{n_0}$ to $M_\mathcal A$. Let $m$ be the 
metric given by pulling back, by $g$, the metric on $Y$ to a metric on $M_\mathcal A$.
For all large enough $n$, define
$$f_n:(M_\mathcal A,P_\mathcal A)\to (Y_n, \partial\MT_{\ep,n}(P_\mathcal A\cap P)\cup \partial\MT_{2\ep,n}(P_\mathcal A\setminus P))$$
to be given by $\psi_n\circ g$. Notice that $f_n$ is then $2$-bilipschitz with respect to $m$ on $M_\mathcal A$.
Since $f_n$ is isotopic to $\hhat s_n$ on ${\rm Fr}(\mathcal A)$ (and pared homotopic on $(M_\mathcal A,P_\mathcal A)$)
and $\hhat s_n$ extends to a pared homeomorphism
between $(M,P)$ and $(C_n,P_n)$, $f_n$ also extends to a pared homeomorphism from $(M,P)$ to $(C_n,P_n)$
which is homotopic to $\hhat s_n$ and hence in the homotopy class of $\rho_n$. The remaining properties
hold by construction and are easily checked. Thus $f_n$ is a model
core map for $(\AAA,m,\ep)$. This concludes the proof of Theorem
\ref{core for sequence}.
\qed

\subsection{The algebraically convergent case}

One can improve on the statement of Theorem \ref{core for sequence} in the case when the sequence
$\{\rho_n\}$ is algebraically convergent. Specifically, 
after passing to a subsequence so that $\{N_{\rho_n}\}$ converges geometrically to a hyperbolic
3-manifold $N$, one may assume that there is a single model core for all large enough $n$,
the model core isometrically embeds in $N$,
the  partial core lifts to $N_\rho$ (where $\rho=\lim\rho_n$), and
the model core map is globally 2-bilipschitz for all large enough $n$.
This improvement will not be used in this paper, but we expect this strengthened form to have applications in future work.

\begin{theorem}{core for convergent sequence}
Suppose that  $(M,P)$ is a pared manifold with pared incompressible boundary, $\ep\in(0,\mu_3/2)$,
and $\{\rho_n\}$ is a sequence in $D(M,P)$, representing points in
$AH_0(M,P)$, and converging to \hbox{$\rho\in D(M,P)$}. 
Moreover suppose $\{\rho_n(\pi_1(M))\}$ converges geometrically to $\Gamma$
and $\pi:N_\rho\to N=\mathbb H^3/\Gamma$ is the covering map.

Then, there  is a metric $m$ on $M$ such that, for all large enough $n$, 
$(\emptyset,m,\epsilon)$ is a model core which controls $\rho_n$
with model core map 
$$f_n:(M,P)\to (N_{\rho_n}\setminus \MT_{\ep,n}(P),\partial\MT_{\ep,n}(P)),$$
i.e. $f_n$
is a 2-bilipschitz embedding  such that $f_n(\partial M)\subset N_{\rho_n}\setminus (N_{\rho_n})_{(0,\ep)}$
and there exists a homeomorphism $\phi_n$ supported on $\Sigma(M,P)$ so that
$f_n\circ\phi_n$ is in the homotopy class of $\rho_n$.

Moreover, there exists a robust collection $\mathcal A$ of essential annuli for $(M,P)$, such that
\begin{enumerate}
\item
$\ell_\rho(\mathcal A)=0$ and if $B$ is  an essential annulus in $(M_\mathcal A,P_\mathcal A)$, then $\ell_\rho(B)>0$.
\item
If $A\in\mathcal A$, then $\pi_1(A)$ is a maximal cyclic subgroup of $\pi_1(M)$.
\item
There exists an embedding $g:M\to N$, isometric with respect to $m$, such that if $R$ is a component
of $M_\AAA$, then
$g|_{R}$ lies  in the homotopy class of $\pi_*\circ \rho|_{\pi_1(R))}$.
\item
If $\hhat M_\AAA$ is obtained from $M_\AAA$ by appending all
components of $M\setminus M_\AAA$ which
contain a component of $P$, then $\phi_n$ is supported on $M\setminus \hhat M_\AAA$ and
the restriction of  $g$ to $\hhat M_\AAA$ lifts to an embedding
in $N_\rho$.
\end{enumerate}
\end{theorem}

Notice that $[\rho]$ need not lie in $AH_0(M)$. In particular, it may not be the case that $(\emptyset,m,\ep)$ is a model
core for $\rho$.

The proof of Theorem \ref{core for convergent sequence} is simpler than that of Theorem \ref{core for sequence}
and is essentially contained in section 8 of Anderson-Canary-McCullough \cite{ACM}, 
but we will explain how to modify our proof of Theorem \ref{core for sequence} to obtain
Theorem \ref{core for convergent sequence}.

\medskip\noindent
{\em Sketch of proof of Theorem \ref{core for convergent sequence}:}
Choose $\mathcal A$ to be a maximal robust collection of essential annuli so that $\ell_\rho(\mathcal A)=0$,
so $\mathcal A$ satisfies (1).

Property (2) will follow quickly from the following topological lemma.

\begin{lemma}{parabolic annuli are primitive}{}
Suppose that $(M,P)$ is a pared 3-manifold with pared incompressible boundary
and $h:(M,P)\to (M_0,P_0)$ is a pared homotopy equivalence.
If $A$ is an essential annulus in $(M,P)$ and $\pi_1(A)$ is not a maximal cyclic subgroup of $M$,
then there exists a root  $\alpha$ of  the generator of $\pi_1(A)$ so that $h(\alpha)$ is not homotopic into the boundary of $M_0$.
\end{lemma}

\begin{proof}
Up to isotopy, the annulus $A$ is either (a) the regular neighborhood of an essential M\"obius band in an interval bundle component
$\Sigma_0$ of $\Sigma(M,P)$ or (b) a component of the frontier of a solid torus component $V$ of $\Sigma(M,P)$ and $\pi_1(A)$
is a proper subgroup of $\pi_1(V)$,
see Johannson \cite[Lemma  32.1]{johannson}. 
Since $(M,P)$ has pared incompressible boundary, $(M_0,P_0)$ also has pared incompressible boundary,
see \cite[Prop. 1.2]{bonahon} or  \cite[Lemma 5.2.1]{canary-mccullough}.
Johannson's Classification Theorem \cite[Thm. 24.2]{johannson} implies that $h$ may be admissibly homotoped so that 
$h^{-1}(\Sigma(M_0,P_0))=\Sigma(M,P)$ and $h$ is a homeomorphism on $\overline{M-\Sigma(M,P)}$.
In case (a), one may apply \cite[Lemma 2.11.3]{canary-mccullough} to admissibly homotope $h$ so that it 
is a homeomorphism of $\Sigma_0$ to an interval bundle component $\Sigma_0$ of $\Sigma(M_0,P_0)$, 
while in case (b),  $h(A)$ is a frontier annulus of a solid torus component $V_0$ of $\Sigma(M_0,P_0)$ so that
$\pi_1(h(A))$ is a proper subgroup of $\pi_1(V_0)$. In either case, the core curve of $A$ has
a root $\alpha$ in $\pi_1(M)$ so that $h_*(\alpha)$ is not represented by a simple closed curve on $\partial M_0$.
\end{proof} 

Now suppose $A$ is an essential annulus in $\AAA$, so $\rho(\pi_1(A))$ is parabolic, since $\ell_\rho(A)=0$.
Let $(M_\rho,Q_\rho)$ be a relative compact core for $N_\rho^\ep$ and let $P_\rho$ be the collection of components
of $Q_\rho$ lying in $\partial\MT(P)$.
Then there exists a pared homotopy equivalence $h_\rho:(M,P)\to (M_\rho,P_\rho)$ in the homotopy
class of $\rho$. If $\pi_1(A)$ is not a maximal cyclic subgroup, then, by Lemma \ref{parabolic annuli are primitive},
there exists a root $\alpha$ of
$\pi_1(A)$ so that $h_*(\alpha)$ is not homotopic into $\partial M_\rho$.
Since $\rho(\pi_1(A))$ is parabolic, $\rho(\alpha)$ is also parabolic.
However, this contradicts the fact that $(M_\rho,Q_\rho)$ is a relative compact core for $N_\rho^\ep$.
Therefore, (2) holds.

\medskip

Let $\mathcal R$ be the collection of components of $(M_\AAA,P_\AAA)$. Proposition \ref{partialcoreD}
provides a disjoint collection $\{ (Y_R,Z_R)\}_{(R,Q)\in\RR}$ of relative compact
carriers for $\{\rho^R(\pi_1(R))\}$ in $N^{2\ep}$.  Let $h_R:(R,Q)\to (Y_R,Z_R)$
be the associated pared homotopy equivalence.

We may adjust $\{ (Y_R,Z_R)\}$ and the maps $\{h_R\}$,
exactly as in the proof of Theorem \ref{core for sequence}, so that if $P_0$ is a component of $P$ contained
in a component $(R,Q)$ of $\mathcal R$, then $h_R(P_0)\subset \partial N^\ep$.
Let $Y=\bigcup Y_R$, $Z=\bigcup Z_R$ and
\hbox{$h=\bigcup h_R:Y\to N^\ep$}.
Let $\eta:Y\to N_\rho$ be the section of $\pi$ over $Y$.

Let $\MT_{2\ep}(\AAA)$ denote the components of $N_{(0,2\ep)}$ associated
to the components of $\AAA$. For each such component $\MT$,
let $B(\MT)$ be an incompressible annulus (if $\MT$ has rank 1) or torus (if not)
containing the intersection of $\boundary \MT$ with $Z$. Let $B(\AAA)$
be the union of these annuli and tori.
We then append  to $Y$ the set of all
points in \hbox{$\MT_{2\ep}(\AAA)\setminus\MT_\ep(\AAA)$} which project radially to $B(\AAA)$. The resulting submanifold
of $N^\ep$ will be called $C$. 
(Notice that property (2) implies that \hbox{$C\cap\MT_{2\ep}(A)$} is always a compact core for $\MT_{2\ep}(A)$. This
fact is what allows us to use a simplified and uniform construction of $C$ in the algebraically convergent case.)
We construct a pared locus $P_C$ for $C$ by first including $h(P\cap M_\AAA)$.
If $P_0$ is a component of $P$ contained  in a solid or thickened torus
component $X$ of $\Sigma(M,P)$ and $A$ is a component of ${\rm Fr}(X)$, then we 
add $\boundary \MT_\ep(A) \cap C$ to $P_C$. In particular, $P_C=\partial N^\ep\cap C$.

For all large enough $n$, there exist 2-bilipschitz maps $\psi_n:C\to N_{\rho_n}$
so that $\psi_n$ takes \hbox{$C\cap \partial N^{\ep}$} to
\hbox{$\boundary N_{\rho_n}^\ep$},
$\psi_n$ takes \hbox{$C\cap  (N\setminus N^{2\ep})$} to \hbox{$(N_{\rho_n}\setminus N_{\rho_n}^{2\ep})$}
and $\psi_n|_{Y_R}$ lies in the homotopy class of $\rho_n\circ\rho^{-1}\circ (\eta|_{Y_R})_*$ for all components $Y_R$ of $Y$.
Let $(C_n,P_n)=(\psi_n(C),\psi_n(P_C))$. One may then, just as in the proof of Theorem \ref{core for sequence},
construct, for all large enough $n$,  
a homeomorphism \hbox{$\hhat s_n:(M,P)\to (C_n,P_n)$} in the homotopy class of $\rho_n$ such that
$\hhat s_n(R)=\psi_n(Y_R)$ for all $(R,Q)\in \RR$.

Let $g:(M,P)\to (C,P_C)$ be given by $(\psi_{n_0}|_C)^{-1}\circ\hhat
s_{n_0}$ (for a fixed large enough $n_0$)
and pull back the metric on $C$ to obtain a metric $m$ on $M$. Then $f_n=\psi_n\circ g$ is a \hbox{2-bilipschitz}
homeomorphism, for all large enough $n$. (Notice that the presence of an embedded copy of $M$ in the geometric
limit allows for a more concrete construction of $f_n$ in the algebraically convergent case.)
Let
$$\phi_n=f_n^{-1}\circ\hhat s_n=\hhat s_{n_0}^{-1}\circ \psi_{n_0}\circ\psi_n^{-1}\circ \hhat s_n$$
and notice that $f_n\circ\phi_n=\hhat s_n$ is in the homotopy class of $\rho_n$ as desired.
By the choice of $\psi_n$ in the previous paragraph, 
the restriction of $\psi_{n_0}\circ\psi_n^{-1}$ to each component of
$\hhat s_n(M_\AAA)$ is in the homotopy class of $\rho_{n_0}\circ \rho_n^{-1}$. In particular, the restriction of 
\hbox{$ \psi_{n_0}\circ\psi_n^{-1}$} to $\hhat s_n(M_\AAA)$ is homotopic to $\hhat s_{n_0}\circ \hhat s_n^{-1}$, so
$\phi_n$ is homotopic, hence admissibly homotopic (see Lemma \ref{paredupgrade}), to the identity on $(M_\AAA,P_\AAA)$. 
Therefore, we may assume that $\phi_n$ is the identity on $M_\AAA$ and thus on $M\setminus\Sigma(M,P)\subset M_\AAA$.
It follows that $f_n$ is a model core map.
Property (3) is immediate from the construction.

Notice that $\eta\circ g$ is a lift of the restriction of $g $ to $M_\AAA$.
Suppose that $V$ is a component of $M\setminus M_\AAA$ which
contains a component $P_0$ of $P$. It follows from \cite[Prop. 2.7]{ELC2} and the fact that $\pi_1(P_0)$ is a
maximal abelian subgroup of $\pi_1(M)$, that $\pi_1(\MT(P_0))$ is the geometric limit of $\{\rho_n(\pi_1(P_0))\}$.
However, one may assume, after conjugating by a compact family of
elements of ${\rm PSL}(2,\mathbb C)$,
that $\rho_n(\pi_1(P_0))$ all lie in the parabolic subgroup
stabilizing $\infty$. Within this subgroup, which is a planar
translation group, algebraic limits equal geometric limits, and so 
$\pi_1(\MT(P_0))=\rho(\pi_1(P_0))$.
Therefore 
$\MT_{2\ep}(P_0)$ lifts homeomorphically to a component of $(N_\rho)_{(0,2\ep)}$,
which implies that $\eta\circ g$ can be extended over $V$. 
Lemma \ref{basic geometric limit facts}(4) then implies that the restriction of $\psi_{n_0}\circ\psi_n^{-1}$
to each component of $\hhat s_n(\hhat M_\AAA)$ is in the homotopy class of $\rho_{n_0}\circ \rho_n^{-1}$.
We may thus assume, by the argument in the previous paragraph, that $\phi_n$ restricts to the
identity on $\hhat M_\AAA$. So property (4) holds as well.
\qed

\section{Proof of the main theorem}
\label{proof}

We are now ready to prove the main theorem, which we state here in the
pared setting: 

\begin{theorem}{main gen pared}
Let $(M,P)$ be a compact, orientable, pared 3-manifold
with incompressible pared boundary.
For each curve $\alpha$ in $\boundary_{nw} (M,P)$, there exists $K=K(\alpha)$ such
that
$$\ell_{\sigma_{(M,P)}(\rho)}(\alpha)\le K$$
for all $\rho\in AH_0(M,P)$.
\end{theorem}

\medskip\noindent
{\bf Proof of Theorem \ref{main gen pared}:}
Let $\alpha$ be a curve in $\boundary_{nw}(M,P)$, and let
$S$ be the component of $\boundary_0(M,P)$ containing $\alpha$.
Fix $r < \mu_3/2$ such that any curve of length $L_h$ intersecting
$\MT_r(\beta)$ is homotopic to a power of $\beta$. 
Applying Theorem \ref{uniform cores} with
$\ep = r$ we see that, given $\rho\in AH_0(M,P)$,  there is a
model core $(\AAA,m,r)$, chosen from a finite list, with an associated
model core map $f$.

Noting that $\alpha$ can be made disjoint from $\AAA$, 
choose a minimal length representative of $\alpha$ in 
\hbox{$\boundary_{0}(M,P)\intersect M_\AAA$} and let $\alpha_f$ be its $f$-image in
$N_\rho$. Its length is bounded by some $L_0$ depending on $\alpha$ and
the finitely many possibilities for $m$.
$\alpha_f$ is contained in
$F = f(S)$ which, by the properties of a model map, is contained
in the $r$-thick part of $N_\rho$. 

Let $\pi:N_S \to N_\rho$ be the cover associated to $\pi_1(S)$ and let
$\hhat F$ be the homeomorphic lift of $F$  to this cover.
Let $\hhat\alpha$ be the lift of $\alpha_f$ to $\hhat F$.
By definition, $\sigma_{(M,P)}^S(\rho)$ is the bottom ending invariant of
$N_S$. 

Let $\mathcal H$ be the invariant system of horoballs for $\rho(\pi_1(M))$ which is the pre-image
of the non-compact components of $N_{(0,r)}$ and let $\hhat{\mathcal H}$ denote the induced subcollection
of invariant horoballs for $\rho(\pi_1(S))$. 
If $C=f(M)$ and $D=f(P)$, then $(C,D)$ is a relative compact
core for $N_\rho^{\mathcal H}$. Moreover, $\hhat F$ is a level surface for $N_S^{\hhat{\mathcal H}}$.

Let $\hhat C$ be the component of $\pi^{-1}(C)$ which contains $\hhat F$ in its boundary.
Then $\hhat C$ lies below $\hhat F$. Suppose that $\beta$ is a curve in $S$ which intersects
$\alpha$ essentially with length $\ell_\rho(\beta)\le L_h$. By our choice of $r$, $\beta^*$ lies in
$N_S^{\hhat{\mathcal H}}$.
We claim  that either $\beta^*$ lies above $\hhat F$ or $\beta^*$ intersects $\hhat C$.

Suppose $\beta^*$ does not lie
above $\hhat F$.
If $\beta^*$ does not intersect $\hhat C$ it must be contained in a component $U$ of 
$N_S\setminus\hhat C$ which lies below $\hhat F$. The component $U$ shares a
boundary component $E$ with $\hhat C$. 
A homotopy of $\beta^*$ to $\hhat F$, intersected with $\hhat C$
and surgered, gives rise to an
immersed essential annulus in $\hhat C$ joining the lift of $f(\beta)$ to a curve in $E$,
and hence to an immersed essential annulus in $C$  joining $f(\beta)$ to a curve in $\partial C$.
Therefore, by Theorem \ref{charprop}(4), $\beta$  is homotopic into the window of $M$, 
so cannot have essential intersection  with $\alpha$. 
This contradiction implies that $\beta^*$ must intersect $\hhat C$.

\realfig{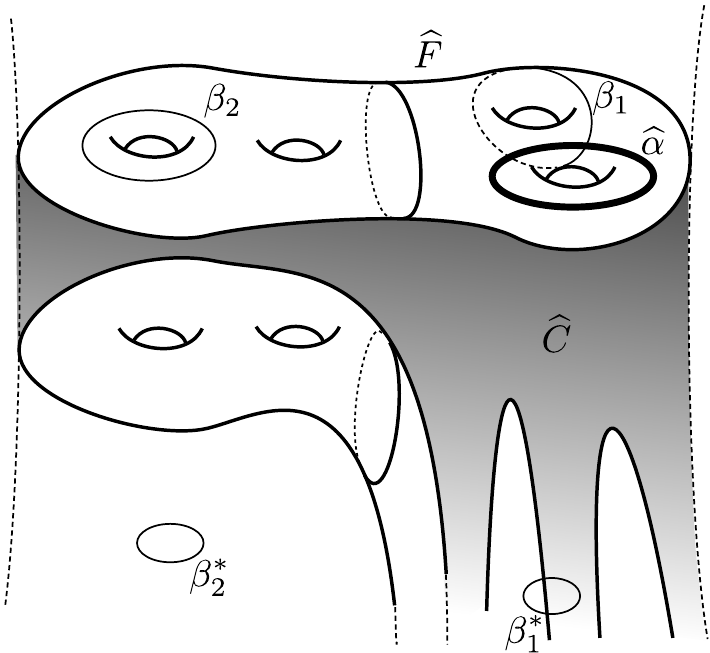}{In $ N_S$, $\beta_1$ intersects $\alpha$ so
  $\beta_1^*$ must meet $\hhat C$ if it doesn't lie above $\hhat F$. But
  $\beta_2$ is in the window so $\beta_2^*$ is not constrained.}

Any curve that does not lie above $\hhat\alpha$ also does not lie above
$\hhat F$. 
We conclude that the geodesic representative of every curve in $\mathcal C(\hhat\alpha,L_h)$ intersects
$\hhat C$.

Notice that if two (homotopically) distinct curves in $S$ are in the
same  homotopy class in $M$, 
then  there is an immersed essential annulus  in $M$ joining them, so both curves are
homotopic into the window (again by Theorem \ref{charprop}(4)).
It follows that neither curve could intersect $\alpha$. Therefore, any two distinct curves in
$\mathcal C(\hhat\alpha,L_h)$ project to (homotopically) distinct curves in $N_\rho$.

So, the geodesic representatives of curves in $\mathcal C(\hhat\alpha,L_h)$ project
to distinct curves which intersect $C$. 
Our choice of $r$ guarantees that  the geodesic representative of any curve in 
$\mathcal C(\hhat\alpha,L_h)$ cannot intersect $\MT_{r}(f(\AAA))$ (if it did
it would be homotopic into $\Sigma(M,P)$,  again by Theorem \ref{charprop}(4), so would not intersect $\alpha$
essentially), 
so it must intersect $f(M_\AAA)$.
Since  each component of $f(M_\AAA)$
has uniformly bounded diameter,
there exists a uniform bound on the number of geodesics
of length at most $L_h$ which can intersect $f(M_\AAA)$. This bounds
the number of elements of $\mathcal C(\hhat\alpha,L_h)$.
Moreover, there is a lower
bound on the length of the geodesic representative of any curve in $\mathcal C(\hhat\alpha,L_h)$,
since if the closed geodesic is very short, its Margulis tube would have very large radius and would thus
be forced to contain an entire component of $f(M_\AAA)$,
which is impossible.
Finally, note that since $F$ is in the $r$-thick part, so is $\hhat F$. This establishes all the hypotheses of 
Theorem \ref{upper bound}, which therefore gives a uniform upper bound
on $\ell_{\sigma_{M,S}(\rho)}(\alpha)$ and completes the proof.
\qed

\bigskip

The generalization of Corollary \ref{maincor gen} to the pared setting
is the following:

\begin{corollary}{maincor gen pared}
Let $(M,P)$ be a compact, orientable, pared 3-manifold
with incompressible pared boundary.
If $W$ is a component of  $\boundary_{nw} (M,P)$, then
the image of $\sigma_{(M,P)}^W$ is bounded in $\mathcal{F}(W)$.
\end{corollary}

\begin{proof}
If $W$ is an annulus then $\FF(W) = [0,\infty)$ and $\sigma_{(M,P)}^W$
  is the map that records the length of the core of $W$ in the
  skinning image. Hence the bound follows immediately from Theorem
  \ref{main gen pared}.

For any non-annular $W$ we note that, since
$\ell_{\sigma_{(M,P)}(\rho)}(\alpha) < \infty$ for each $\alpha$ in $W$,
$W$ must be contained in a geometrically finite component of 
$\sigma_{(M,P)}(\rho)$ (in fact this statement is explicitly
demonstrated in the first step of the proof of Theorem \ref{upper
  bound}). Thus, lifting to the cover associated to $W$ we are able to
define $\sigma_{(M,P)}^W(\rho)\in\FF(W)$.
Now for any $X\in\FF(W)$ an upper
bound on the lengths in $X$ of a filling
system of non-peripheral curves in $W$ restricts $X$ to a compact
subset of $\FF(W)$. Thus again Theorem   \ref{main gen pared} implies
our statement.
\end{proof}

\end{document}